\documentclass[fontsize=12pt,a4paper,headings=normal,
twoside=false,leqno,parskip=half-,abstract=true]{scrartcl}
\usepackage[english]{babel}
\usepackage[utf8]{inputenc}
\usepackage{hyperref}
\hypersetup{
 pdftitle={Sturm 3-ball global attractors 3: Examples},
 pdfauthor={Bernold Fiedler, Carlos Rocha},
 colorlinks=true,
 linkcolor=blue,
 citecolor=blue,
 filecolor=blue,
 urlcolor=blue}
  
 \usepackage{afterpage}

\usepackage{graphicx}
\usepackage[format=plain,labelfont=bf,font=small]{caption}
\usepackage{xcolor}
\usepackage[arrow, matrix, curve]{xy}
\usepackage{float}

\usepackage{caption}
\usepackage{tabulary}
\usepackage{array}
\newcolumntype{N}[1]{>{\centering\arraybackslash}m{#1}}

\usepackage{amsmath,amsthm}
\swapnumbers 
\usepackage{amssymb} 
\newcommand{\transv}{\mathrel{\text{\tpitchfork}}}
\makeatletter
\newcommand{\tpitchfork}{%
  \vbox{
    \baselineskip\z@skip
    \lineskip-.52ex
    \lineskiplimit\maxdimen
    \m@th
    \ialign{##\crcr\hidewidth\smash{$-$}\hidewidth\crcr$\pitchfork$\crcr}
  }%
}
\makeatother
\usepackage{latexsym}

\usepackage[notref,notcite,color,final 
]{showkeys}

\definecolor{refkey}{rgb}{1,0,0}
\definecolor{labelkey}{rgb}{1,0,0}

\usepackage{tikz}

  \mathchardef\ordinarycolon\mathcode`\:
  \mathcode`\:=\string"8000
  \begingroup \catcode`\:=\active
    \gdef:{\mathrel{\mathop\ordinarycolon}}
  \endgroup

\theoremstyle{plain}
\newtheorem{thm}{Theorem}[section]
\newtheorem{lem}[thm]{Lemma}

\newtheorem{prop}[thm]{Proposition}
\newtheorem{cor}[thm]{Corollary}
\newtheorem{defi}[thm]{Definition}

\begin{document}

\title{\LARGE{Boundary orders and  geometry\\
of the signed Thom-Smale complex\\
for Sturm global attractors}
\vspace{1cm}}
{\subtitle{	
	\vspace{1ex}
	{\large -- Dedicated to the dear memory of Geneviève Raugel --}}\vspace{1ex}
	}

\author{
 \\
{~}\\
Bernold Fiedler* and Carlos Rocha**\\
\vspace{2cm}}

\date{version of \today}
\maketitle
\thispagestyle{empty}

\vfill

*\\
Institut für Mathematik\\
Freie Universität Berlin\\
Arnimallee 3\\ 
14195 Berlin, Germany\\
\\
**\\
Center for Mathematical Analysis, Geometry and Dynamical Systems\\
Instituto Superior T\'ecnico\\
Universidade de Lisboa\\
Avenida Rovisco Pais\\
1049--001 Lisbon, Portugal\\


\newpage
\pagestyle{plain}
\pagenumbering{roman}
\setcounter{page}{1}

\begin{abstract}
\noindent
We embark on a detailed analysis of the close relations between combinatorial and geometric aspects of the scalar parabolic PDE
	\begin{equation}\label{eq:*}
	u_t = u_{xx} + f(x,u,u_x) \tag{$*$}
	\end{equation}
on the unit interval $0 < x<1$ with Neumann boundary conditions.
We assume $f$ to be dissipative with $N$ hyperbolic equilibria $v\in\mathcal{E}$.
The global attractor $\mathcal{A}$ of \eqref{eq:*}, also called \emph{Sturm global attractor}, consists of the unstable manifolds of all equilibria $v$.
As cells, these form the \emph{Thom-Smale complex} $\mathcal{C}$.
\medskip

\noindent
Based on the fast unstable manifolds of $v$, we introduce a refinement $\mathcal{C}^s$ of the regular cell complex $\mathcal{C}$, which we call the \emph{signed Thom-Smale complex}.
Given the signed cell complex $\mathcal{C}^s$ and its underlying partial order, only, we derive the two total boundary orders $h_\iota:\{1,\ldots , N\}\rightarrow\mathcal{E}$ of the equilibrium values $v(x)$ at the two Neumann boundaries $\iota=x=0,1$.
In previous work we have already established how the resulting Sturm permutation 
\[\sigma:=h_{0}^{-1} \circ h_1,\]
conversely, determines the global attractor $\mathcal{A}$ uniquely, up to topological conjugacy.

\end{abstract}

\vspace{2cm}
\tableofcontents


\newpage
\pagenumbering{arabic}
\setcounter{page}{1}

\section{Introduction}
\label{sec1}

\numberwithin{equation}{section}
\numberwithin{figure}{section}
\numberwithin{table}{section}

For our general introduction we first follow \cite{firo3d-1, firo3d-2, firo3d-3} and the references there.
\emph{Sturm global attractors} $\mathcal{A}_f$ are the global attractors of scalar parabolic equations
	\begin{equation}
	u_t = u_{xx} + f(x,u,u_x)
	\label{eq:1.1}
	\end{equation}
on the unit interval $0<x<1$.
Just to be specific we consider Neumann boundary conditions $u_x=0$ at $x \in \{0,1\}$.
Standard theory of strongly continuous semigroups provides local solutions $u(t,x)$ in suitable Sobolev spaces $u(t, \cdot) \in X \subseteq C^1 ([0,1], \mathbb{R})$, for $t \geq 0$ and given initial data $u=u_0(x)$ at time $t=0$.

We assume the solution semigroup $u(t,\cdot)$ generated by the nonlinearity $f \in C^2$ to be \emph{dissipative}: there exists some large constant $C$, independent of initial conditions, such that any solution $u(t,\cdot)$ satisfies $\|u(t,\cdot)\|_X\leq C$ for all large enough times $t\geq t_0(u_0)$.
In other words, any solution $u(t,\cdot)$ exists globally in forward time $t\geq 0$, and eventually enters a fixed large ball in $X$.
Explicit sufficient, but by no means necessary, conditions on $f=f(x,u,p)$ which guarantee dissipativeness are sign conditions $f(x,u,0)\cdot u<0$, for large $|u|$, together with subquadratic growth in $|p|$.

For large times $t\rightarrow\infty$, the large attracting ball of radius $C$ in $X$ limits onto the maximal compact and invariant subset $\mathcal{A}=\mathcal{A}_f$ of $X$ which is called the \emph{global attractor}. 
Invariance refers to, both, forward and backward time.
In general, the global attractor $\mathcal{A}$ consists of all solutions $u(t,\cdot)$ which exist globally, for all positive and negative times $t\in\mathbb{R}$, and remain bounded in $X$.
See \cite{he81, pa83, ta79} for a general PDE background, and \cite{bavi92, chvi02, edetal94, ha88, haetal02, la91, ra02, seyo02, te88} for global attractors in general.

For Geneviève Raugel, in particular, global attractors were a main focus of interest.
Her beautiful survey \cite{ra02}, for example, puts our past and present work on scalar one-dimensional parabolic equations in a much broader perspective.

For the convenience of the reader, we provide a rather complete background on our current understanding of the global attractors of \eqref{eq:1.1}. 
It is not required, and would in fact be pedantic, to read all technical references given. Rather, the present paper is elementary, although nontrivial, given the background facts which we will now summarize.

Equilibria $u(t,x) = v(x)$ are time-independent solutions, of course, and hence satisfy the ODE
	\begin{equation}
	0 = v_{xx} + f(x,v,v_x)
	\label{eq:1.2}
	\end{equation} 
for $0\leq x \leq 1$, again with Neumann boundary.
Here and below we assume that all equilibria $v$ of \eqref{eq:1.1}, \eqref{eq:1.2} are \emph{hyperbolic}, i.e. without eigenvalues $\lambda=0$ of their Sturm-Liouville linearization
\begin{equation}
\label{eq:SL}
\lambda u = u_{xx} + f_p(x,v(x),v_x(x)) u_x + f_u(x,v(x),v_x(x)) u
\end{equation}
under Neumann boundary conditions.
We recall here that all eigenvalues $\lambda_0>\lambda_1>\ldots$ are algebraically simple and real.
The \emph{Morse index} $i(v)$ of $v$ counts the number of unstable eigenvalues $\lambda_j>0$.
In other words, the Morse index $i(v)$ is the dimension of the unstable manifold $W^u(v)$ of $v$.
Let $\mathcal{E} = \mathcal{E}_f \subseteq \mathcal{A}_f$ denote the set of equilibria.
Our generic hyperbolicity assumption and dissipativeness of $f$ imply that $N$:= $|\mathcal{E}_f|$ is odd.

It is known that \eqref{eq:1.1} possesses a Lyapunov~function, alias a variational or gradient-like structure, under separated boundary conditions;  see \cite{ze68, ma78, mana97, hu11, fietal14, lafi18}. 
In particular, the time invariant global attractor consists of equilibria and of solutions $u(t, \cdot )$, $t \in \mathbb{R}$, with forward and backward limits, i.e.
	\begin{equation}
	\underset{t \rightarrow -\infty}{\mathrm{lim}} u(t, \cdot ) = v\,,
	\qquad
	\underset{t \rightarrow +\infty}{\mathrm{lim}} u(t, \cdot ) = w\,.
	\label{eq:1.3}
	\end{equation}
In other words, the $\alpha$- and $\omega$-limit sets of $u(t,\cdot )$ are two distinct equilibria $v$ and $w$.
We call $u(t, \cdot )$ a \emph{heteroclinic} or \emph{connecting} orbit, or \emph{instanton},  and write $v \leadsto w$ for such heteroclinically connected equilibria. 
See fig.~\ref{fig:1.0}(a) for a modest 3-ball example with $N=9$ equilibria.

\begin{figure}[p!]
\centering \includegraphics[width=0.9\textwidth]{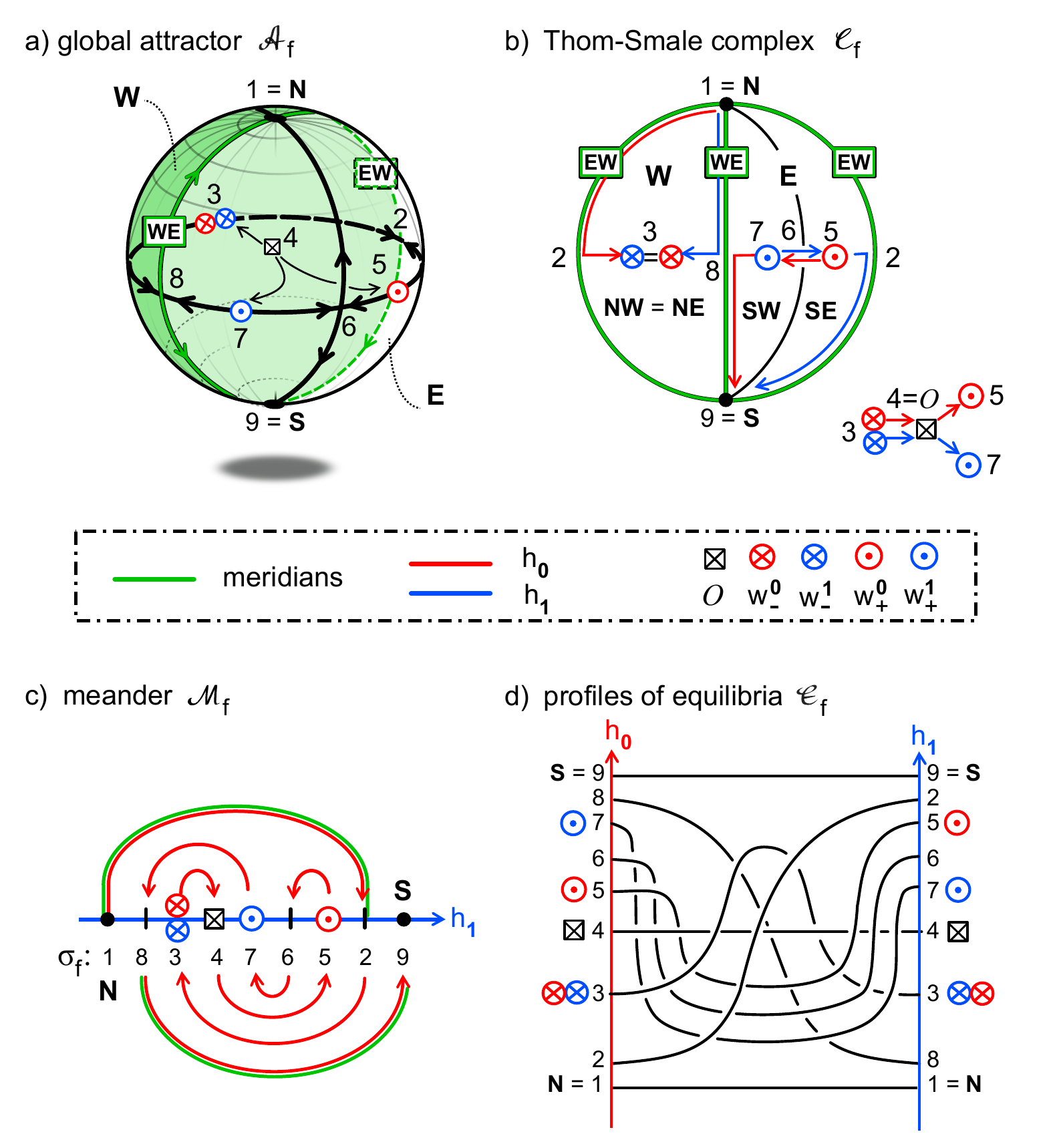}
\caption{\emph{
Example of a Sturm 3-ball global attractor $\mathcal{A}_f = clos W^u(\mathcal{O})$. Equilibria are labeled as $\mathcal{E}=\{1,\ldots,9\}$. The previous papers \cite{firo3d-1, firo3d-2} established the equivalence of the viewpoints (a)--(d).
(a) The Sturm global attractor $\mathcal{A}$, 3d view, including the location of the poles $\textbf{N}$, $\textbf{S}$, the (green) meridians $\textbf{WE}$, $\textbf{EW}$, the central equilibrium $\mathcal{O}=4$ and the hemispheres $\textbf{W}$ (green), $\textbf{E}$.
(b) The Thom-Smale complex $\mathcal{C}_f$ of the boundary sphere $\Sigma^2 = \partial c_\mathcal{O}$, including the Hamiltonian SZS-pair of paths $(h_0,h_1)$, (red/blue). 
The right and left boundaries denote the same $\textbf{EW}$ meridian and have to be identified. 
(c) The Sturm meander $\mathcal{M}$ of the global attractor $\mathcal{A}$. 
The meander $\mathcal{M}$ is the curve $a \mapsto (v,v_x)$, at $x=1$, which results from Neumann initial conditions $(v,v_x)=(a,0)$, at $x=0$, by shooting via the equilibrium ODE \eqref{eq:1.2}. 
Intersections of the meander with the horizontal $v$-axis indicate equilibria. 
(d) Spatial profiles $x\mapsto v(x)$ of the equilibria $v \in \mathcal{E}$. 
Note the different orderings of $v(x)$, by $h_0 = \mathrm{id}$ at the left boundary $x=0$, and by the Sturm permutation $\sigma = h_1 = (1\ 8\ 3\ 4\ 7\ 6\ 5\ 2\ 9)$ at the right boundary $x=1$. 
The same orderings characterize the meander $\mathcal{M}$ in (c) and the Hamiltonian SZS-pair $(h_0,h_1)$ in the Thom-Smale complex (b).
}}
\label{fig:1.0}
\end{figure}

We attach the name of \emph{Sturm} to the PDE \eqref{eq:1.1}, and to its global attractor $\mathcal{A}_f$\,. This refers to a crucial \emph{Sturm nodal property} of its solutions, which we express by the \emph{zero number} $z$.
Let $0 \leq z (\varphi) \leq \infty$ count the number of (strict) sign changes of continuous spatial profiles $\varphi : [0,1] \rightarrow \mathbb{R}, \, \varphi \not\equiv 0$.
For any two distinct solutions $u^1$, $u^2$ of \eqref{eq:1.1}, the zero number
	\begin{equation}
	t \quad \longmapsto \quad z(u^1(t, \cdot ) - u^2(t, \cdot ))\,
	\label{eq:1.4}
	\end{equation}
is then nonincreasing with time $t$, for $t\geq0$, and finite for $t>0$.
Moreover $z$ drops strictly with increasing $t>0$, at any multiple zero of the spatial profile $x \mapsto u^1(t_0 ,x) - u^2(t_0 ,x)$; see \cite{an88}.
See Sturm \cite{st1836} for the linear autonomous variant \eqref{eq:SLpar} below. 

For example, let $\varphi_j$ denote the $j$-th Sturm-Liouville eigenfunction $\varphi_j$ of the linearization at any equilibrium $v$.
Sturm not only observed that $z(\varphi_j)=j$.
Already in 1836 he proved the much more general statement
	\begin{equation}
	\label{eq:1.5a}
	0 \neq \varphi \in \mathrm{span}\,\{\varphi_j\,,\varphi_{j+1}\,\ldots,\varphi_k\} \quad\Longrightarrow\quad j\leq z(\varphi)\leq k\,.
	\end{equation} 
His proof was based on the solution $\psi(t,\cdot)$ of the associated linear parabolic equation
	\begin{equation}
	\label{eq:SLpar}
	\psi_t = \psi_{xx} + f_p(x,v(x),v_x(x))\,\psi_x + f_u(x,v(x),v_x(x))\,\psi\,,
	\end{equation} 
under Neumann boundary conditions and with initial condition $\psi(0,\cdot)=\varphi$.
He then invoked nonincrease of the zero number $t\mapsto z(\psi(t,\cdot))$.
Since the rescaled limits of $\psi$ for $t\rightarrow\pm\infty$ are eigenfunctions, this proved his claim \eqref{eq:1.5a}.

As a convenient notational variant of the zero number $z$, we will also write
	\begin{equation}
	z(\varphi) = j_{\pm}
	\label{eq:1.5}
	\end{equation}
to indicate $j$ strict sign changes of $\varphi$, by $j$, and the sign $\pm \varphi (0) >0$, by the index $\pm$.
For example, we may fix the sign of any $j$-th Sturm-Liouville eigenfunction $\varphi_j$ such that $\varphi_j(0)>0$, i.e.~$z(\varphi_j) = j_+$\,.

The consequences of the Sturm nodal property \eqref{eq:1.4} for the nonlinear dynamics of \eqref{eq:1.1} are enormous.
For an introduction see \cite{ma82, brfi88, fuol88, mp88, brfi89, ro91, fisc03, ga04} and the many references there.
Let us also mention Morse-Smale transversality, a prominent concept in \cite{pasm70, pame82, ol83}.
The Sturm property \eqref{eq:1.4} automatically implies Morse-Smale transversality, for hyperbolic equilibria.
More precisely, intersections of unstable and stable manifolds $W^u(v)$ and $W^s(w)$ along heteroclinic orbits $v \leadsto w$ are  automatically transverse:
$W^u(v) \transv W^s(w)$.
See \cite{he85, an86}.
In the Morse-Smale setting, Henry \cite{he85} also observed
	\begin{equation}
	\label{eq:1.5b}
	v \leadsto w \quad\Longleftrightarrow\quad w\in\partial W^u(v)\,.
	\end{equation} 
Here $\partial W^u(v) := \mathrm{clos}\,W^u(v) \setminus W^u(v)$ denotes the topological boundary of the unstable manifold $W^u(v)$.

In a series of papers, based on the zero number, we have given a purely combinatorial description of Sturm global attractors $\mathcal{A}_f$; see \cite{firo96, firo99, firo00}.
Define the two \emph{boundary orders} $h_0, h_1$: $\lbrace 1, \ldots, N \rbrace \rightarrow \mathcal{E}$ of the equilibria such that
	\begin{equation}
	h_\iota (1) < h_\iota (2) < \ldots < h_\iota (N) \qquad \mathrm{at}
	\qquad x=\iota \in \{0,1\}\,.
	\label{eq:1.6}
	\end{equation}
See fig.~\ref{fig:1.0}(d) for an illustration with $N=9$ equilibrium profiles, $\mathcal{E} = \{1,\ldots,9\}, \ h_0 = \mathrm{id},\ h_1 = (1\ 8\ 3\ 4\ 7\ 6\ 5\ 2\ 9)$.	
The general combinatorial description of Sturm global attractors $\mathcal{A}$ is based on the \emph{Sturm~permutation} $\sigma \in S_N$ which was introduced by Fusco and Rocha in \cite{furo91} and is defined as
	\begin{equation}
	\sigma:= h_0^{-1} \circ h_1\,.
	\label{eq:1.7a}
	\end{equation}
Already in \cite{furo91}, the following explicit recursions have been derived for the Morse indices $i_k:=i(h_0(k))$:
	\begin{equation}
	\begin{aligned}
	i_1 &:=  i_N := 0\,,\\
	i_{k+1} &:= i_k+(-1)^{k+1}\,
	\mathrm{sign}\, (\sigma^{-1}(k+1)-\sigma^{-1}(k))\,.\\
	\end{aligned}
	\label{eq:1.7b}
	\end{equation}
Similarly, the (unsigned) zero numbers $z_{jk} := z(v_j-v_k)$ are given recursively, for $j\neq k \neq j+1$, as
	\begin{equation}
	\begin{aligned}
	z_{kk} &:= i_k\,,\qquad\qquad\qquad \\z_{1k} &:= z_{Nk}:=0\,,\\
	z_{j+1,k} &:= \scalebox{0.97}[1.0]
	{$ z_{jk} + \tfrac{1}{2}(-1)^{j+1}
	\cdot\left[ \mathrm{sign}\,\left(\sigma^{-1}(j+1)-\sigma^{-1}(k)\right)-\mathrm{sign}\,\left(\sigma^{-1}(j)-\sigma^{-1}(k)\right)\right].
	$}
	\end{aligned}
	\label{eq:1.7c}
	\end{equation}

Using a shooting approach to the ODE boundary value problem \eqref{eq:1.2}, the Sturm~permutations $\sigma \in S_N$ have been characterized, purely combinatorially, as \emph{dissipative Morse meanders} in \cite{firo99}.
Here the \emph{dissipativeness} property requires fixed $\sigma(1)=1$ and $\sigma(N)=N$.
The \emph{Morse} property requires nonnegative Morse indices $i_k\geq0$ in \eqref{eq:1.7b}, for all $k$.
The \emph{meander} property, finally, requires the formal path $\mathcal{M}$ of alternating upper and lower half-circles defined by the permutation $\sigma$, as in fig.~\ref{fig:1.0}(c), to be Jordan, i.e.~non-selfintersecting.
See the beautifully illustrated book \cite{ka17} for ample material on many additional aspects of meanders.

In \cite{firo96} we have shown how to determine which equilibria $v, w$ possess a heteroclinic orbit connection \eqref{eq:1.3}, explicitly and purely combinatorially from dissipative Morse meanders $\sigma$. 
This was based, in particular, on the results \eqref{eq:1.7b} and \eqref{eq:1.7c} of \cite{furo91}.

More geometrically, global Sturm attractors $\mathcal{A}_f$ and $\mathcal{A}_g$ of nonlinearities $f, g$ with the same Sturm permutation $\sigma_f = \sigma_g$ are $C^0$ orbit-equivalent \cite{firo00}.
Only locally, i.e.~for $C^1$-close dissipative nonlinearities $f$ and $g$, this global rigidity result is based on the Morse-Smale transversality property mentioned above. See for example \cite{pasm70, pame82, ol83}, for such local aspects.

More recently, we have pursued a more explicitly geometric approach.
Let us consider \emph{finite~regular cell complexes}
	\begin{equation}
	\mathcal{C} = \bigcup\limits_{v\in \mathcal{E}} c_v\,,
	\label{eq:1.8a}
	\end{equation}
i.e. finite disjoint unions of \emph{cell~interiors} $c_v$ with additional gluing properties of their boundaries.
We think of the labels $v\in \mathcal{E}$ as \emph{barycenter} elements of $c_v$\,.
For cell complexes we require the closures $\bar{c}_v$ in $\mathcal{C}$ to be the continuous images of closed unit balls $\bar{B}_v$ under \emph{characteristic maps}.
We call $\mathrm{dim}\,\bar{B}_v$ the dimension of the (open) cell $c_v$\,. 
For positive dimensions of $\bar{B}_v$ we require $c_v$ to be the homeomorphic images of the interiors $B_v$\,. 
For dimension zero we write $B_v := \bar{B}_v$ so that any 0-cell $c_v= B_v$ is just a point.
The \emph{m-skeleton} $\mathcal{C}^m$ of $\mathcal{C}$ consists of all cells of dimension at most $m$.
Gluing requires $\partial c_v := \bar{c}_v \setminus c_v \subseteq \mathcal{C}^{m-1}$ for any $m$-cell $c_v$\,.
Thus, the boundary $(m-1)$-sphere $S_v := \partial B_v = \bar{B}_v \setminus B_v$ of any $m$-ball  $B_v$\,, $m>0$, maps into the $(m-1)$-skeleton,
	\begin{equation}
	\partial B_v \quad \longrightarrow \quad \partial c_v \subseteq \mathcal{C}^{m-1}\,,
	\label{eq:1.9a}
	\end{equation}
for the $m$-cell $c_v$, by restriction of the characteristic map.
The continuous map \eqref{eq:1.9a} is called the \emph{attaching} (or \emph{gluing}) \emph{map}.
For \emph{regular} cell complexes, more strongly, the characteristic maps $ \bar{B}_v  \rightarrow  \bar{c}_v $ are required to be homeomorphisms, up to and including the \emph{attaching} (or \emph{gluing}) \emph{homeomorphism} \eqref{eq:1.9a} on the boundary $\partial B_v$\,. 
The $(m-1)$-sphere $\partial{c_v}$  is also required to be a sub-complex of $\mathcal{C}^{m-1}$. 
See \cite{frpi90} for some further background on this terminology.

The disjoint dynamic decomposition
	\begin{equation}
	\mathcal{A}_f = \bigcup\limits_{v \in \mathcal{E}_f} W^u(v) =: \mathcal{C}_f
	\label{eq:1.10a}
	\end{equation}
of the global attractor $\mathcal{A}_f$ into unstable manifolds $W^u(v)$ of equilibria $v$ is called the \emph{Thom-Smale complex} or \emph{dynamic complex}; see for example \cite{fr79, bo88, bizh92}.
In our Sturm setting \eqref{eq:1.1} with hyperbolic equilibria $v \in \mathcal{E}_f$\,, the Thom-Smale complex is a finite regular cell complex.
The open cells $c_v$ are the unstable manifolds $W^ u (v)$ of the equilibria $v \in \mathcal{E}_f$\,.
The proof follows from the Schoenflies result of \cite{firo13}; see \cite{firo14} for a summary.

We can therefore define the \emph{Sturm~complex} $\mathcal{C}_f$ to be the regular Thom-Smale complex $\mathcal{C}$
of the Sturm global attractor $\mathcal{A}=\mathcal{A}_f$\,, provided all equilibria $v \in \mathcal{E}_f$ are hyperbolic.
Again we call the equilibrium $v \in \mathcal{E}_f$ the \emph{barycenter} of the cell $c_v=W^u(v)$.
The dimension of $c_v$ is the Morse index $i(v)$, of course.
A 3-dimensional Sturm complex $\mathcal{C}_f$\,, for example, is the regular Thom-Smale complex of a 3-dimensional $\mathcal{A}_f$\,, i.e. of a Sturm global attractor for which all equilibria $v \in \mathcal{E}_f$ have Morse indices $i(v) \leq 3$.
See fig.~\ref{fig:1.0}(b) for the Sturm complex $\mathcal{C}_f$ of the Sturm global attractor $\mathcal{A}_f$ sketched in fig.~\ref{fig:1.0}(a), which is the closure $\bar{c}_\mathcal{O}$ of a single 3-cell.
With this identification we may henceforth omit the explicit subscripts $f$, when the context is clear.

We can now formulate the \emph{main task} of the present paper:\\
\emph{Let the Thom-Smale complex $\mathcal{C}=\mathcal{C}_f$ of a Sturm global attractor $\mathcal{A}=\mathcal{A}_f$ be given, as an abstract regular cell complex. Derive the possible orders $h_\iota:\{1, \ldots, N\}\rightarrow \mathcal{E}$ of the equilibria $v\in \mathcal{E}=\mathcal{E}_f$, evaluated at the boundaries $x=\iota=0,1$.}

Once again: in the example of fig.~\ref{fig:1.0}, this task requires to derive the red and blue boundary orders in (d) from the given complex (b) -- of course without any previous knowledge of the red and blue path cheats in (b) which indicate precisely those orders.

For $\dim \mathcal{A} = 1$, the answer is almost trivial: any heteroclinic orbit is monotone, and therefore $h_0=h_1$\,.
We have also solved this task for Sturm global attractors $\mathcal{A}$ of dimension
	\begin{equation}
	\dim \mathcal{A} = \max\{i(v)\mid v \in \mathcal{E}\} 
	\label{eq:1.7}
	\end{equation}
equal to two; see the planar trilogy \cite{firo08, firo3d-2, firo3d-3}.
For Sturm 3-balls $\mathcal{A}=\bar{c}_ \mathcal{O}$, which are the closure of the unstable manifold cell $c_\mathcal{O}$ of a single equilibrium $\mathcal{O}$ of maximal Morse index $i (\mathcal{O})=3$, our solution has been presented in the 3-ball trilogy \cite{firo3d-1, firo3d-2, firo3d-3}. 
See \eqref{eq:1.12}--\eqref{eq:1.14} below and section \ref{sec5} for further discussion.
The present paper settles the general case.

It has turned out that the Thom-Smale complex $\mathcal{C}=\mathcal{C}_f$ does not determine the boundary orders $h_\iota$ uniquely -- not even when the trivial equivalences \eqref{eq:1.27} below are taken into account.
See \cite{fi94, firo96}.
Our unique construction of $h_\iota$ therefore involves a refinement of the Thom-Smale complex $\mathcal{C}=\mathcal{C}_f$ which we introduce next: the signed Thom-Smale complex $\mathcal{C}^s=\mathcal{C}_f^s$\,.
The examples in \cite{fi94, firo96, firo3d-3} show that the same abstract regular cell complex $\mathcal{C}$ may possess one or several such refinements, as a signed Thom-Smale complex $\mathcal{C}^s=\mathcal{C}_f^s$ of Sturm type -- or no such refinement at all.

The refinement is crucially based on the disjoint \emph{signed hemisphere decomposition}
	\begin{equation}
	\partial W^u(v) =
	\bigcup\limits_{0\leq j< i(v)}^\centerdot
	\Sigma_\pm^j(v)
	\label{eq:1.8}
	\end{equation}
of the topological boundary $\partial W^u= \partial c_v = \bar{c}_v \smallsetminus c_v$ of the unstable manifold $W^u(v)=c_v$\,, for any equilibrium $v$.
As in \cite[(1.19)]{firo3d-2} we define the open hemispheres by their Thom-Smale cell decompositions
	 \begin{equation}
	 \Sigma_ \pm^j(v) :=
	 \bigcup\limits_{w\in \mathcal{E}_\pm^j(v)}^\centerdot
	 W^u(w)
	 \label{eq:1.9}
	 \end{equation}
with the nonempty equilibrium sets
	\begin{equation}
	\mathcal{E}_\pm^j(v) := \lbrace w\in \mathcal{E}_f\,|\,z(w-v)=j_\pm \ \mathrm{and} \
  v\leadsto w\rbrace\,,
	\label{eq:1.10}
	\end{equation}
as barycenters, for $0\leq j<i(v)$.
Equivalently, we may define the hemisphere decompositions, inductively, via the topological boundary $j$-spheres $\Sigma^j(v)$ of the fast unstable manifolds $W^{j+1}(v)$, as
	\begin{equation}
	\Sigma^j(v) =:
	\bigcup\limits_{0\leq k\leq j}^\centerdot \Sigma_\pm^k(v) \,.
	\label{eq:1.11}
	\end{equation}

Here $W^{j+1}(v)$ is tangent to the eigenvectors $\varphi_0, \ldots ,\varphi_j$ of the first $j+1$ unstable Sturm-Liouville eigenvalues $\lambda_0 > \ldots > \lambda_j >0$ of the linearization at the equilibrium $v$.
In fact $\Sigma^{j-1}(v)$ becomes an equator in $\Sigma^j(v)$, recursively, defining the two remaining hemispheres $\Sigma_\pm^j(v)$ in $\Sigma^j(v)$.
See \cite{firo3d-1} for further details.

We call the resulting refined Thom-Smale regular cell complex $\mathcal{C}=\mathcal{C}_f$ of a Sturm global attractor $\mathcal{A}=\mathcal{A}_f$\,, together with the above hemisphere decompositions of all cell boundaries, the \emph{signed Thom-Smale complex} $\mathcal{C}^s=\mathcal{C}_f^s$\,.

Abstractly, we define a \emph{signed hemisphere complex} $\mathcal{C}^s$ via a regular refinement of a given regular cell complex $\mathcal{C}$ as follows.
We recursively bisect each closed $n$-cell $\bar{c}=\bar{c}^n$, equatorially, by closed cells $\bar{c}^j$ of successively lower dimensions  $j=n-1,\ldots, 1$.
(The cell interiors $c^j$ are an abstraction of the fast unstable manifolds $W^j$ of the barycenter of $v$ with Morse index $i(v)=n$.)
On each boundary sphere $\Sigma^j = \partial c^{j+1}$, this induces a decomposition of $\Sigma^j \setminus \Sigma^{j-1}$ into two hemispheres.
For bookkeeping, we may assign signs $\pm$ to these hemispheres and denote them as $\Sigma_\pm^j$\,, respectively.
Without further discussion of proper bookkeeping constraints, we call any such resulting refinement of the original regular cell complex $\mathcal{C}$ a signed hemisphere complex $\mathcal{C}^s$.
It remains a main open question to properly describe the specific sign assignments which characterize the signed Thom-Smale complexes $\mathcal{C}^s=\mathcal{C}_f^s$ arising from the hemisphere decomposition \eqref{eq:1.8}--\eqref{eq:1.11} above.

Consider 3-ball Sturm attractors $\mathcal{A}=\bar{c}_\mathcal{O}$, for example, with Morse index $i(\mathcal{O})=3$. 
Then the signed hemisphere decomposition \eqref{eq:1.11} at $v=\mathcal{O}$ reads
	\begin{equation}
	\Sigma^2(\mathcal{O}) = \partial W^u (\mathcal{O}):=\mathrm{clos}\,W^ u(\mathcal{O})\setminus W^u(\mathcal{O})=
	\bigcup\limits_{j=0}^2 \Sigma_\pm^j\ (\mathcal{O})\,.
	\label{eq:1.12}
	\end{equation}
Here the \emph{North pole} $\mathbf{N}:=\Sigma^0_- (\mathcal{O})$ and the \emph{South pole} $\mathbf{S}:=\Sigma^0_+ (\mathcal{O})$ denote the boundary of the one-dimensional fastest unstable manifold $W^1 = W^1(\mathcal{O})$, tangent to the positive eigenfunction $\varphi_0$ of the largest eigenvalue $\lambda_0$ at $\mathcal{O}$.
Indeed, solutions $t \mapsto u(t,x)$ in $W^1$ are monotone in $t$, 
for any $x$. Accordingly 
	\begin{equation}
	z(\mathbf{N}- \mathcal{O}) = 0_-\,, 
	\quad z(\mathbf{S}- \mathcal{O}) = 0_+\,,
	\label{eq:1.13}
	\end{equation}
i.e. $\mathbf{N} < \mathcal{O} < \mathbf{S}$ for all $0\leq x\leq 1$.
The poles $\mathbf{N},\mathbf{S}$ split the boundary circle $\Sigma^1 = \partial W^2 (\mathcal{O})$ of the 2-dimensional fast unstable manifold $W^2=W^2(\mathcal{O})$ into the two \emph{meridian} half-circles $\mathbf{EW}:= \Sigma^1_- (\mathcal{O})$ and $\mathbf{WE}:=\Sigma^1_+(\mathcal{O})$.
The boundary circle $\Sigma^1$, in turn, splits the boundary sphere $\Sigma^2 = \partial W^u(\mathcal{O})$ of the whole 3-dimensional unstable manifold $W^u=W^3$ of $\mathcal{O}$ into the Western hemisphere $\mathbf{W}:=\Sigma^2_-(\mathcal{O})$ and the Eastern hemisphere $\mathbf{E}:=\Sigma^2_+(\mathcal{O})$.
Omitting the explicit references to the central equilibrium $\mathcal{O}$, the hemisphere translation table becomes:
	\begin{equation}
	\begin{aligned}
	(\Sigma_-^0, \Sigma_+^0) \quad &\mapsto \quad (\mathbf{N}, \mathbf{S})\,;\\
	(\Sigma_-^1, \Sigma_+^1) \quad &\mapsto \quad 
	(\mathbf{EW}, 		\mathbf{WE})\,;\\
	(\Sigma_-^2, \Sigma_+^2) \quad &\mapsto \quad (\mathbf{W}, \mathbf{E})\,.
	\end{aligned}
	\label{eq:1.14}
	\end{equation}
In this case,  a complete characterization of the signed Thom-Smale complexes $\mathcal{C}^s=\mathcal{C}_f^s$ has been achieved. See section \ref{sec5}, definition \ref{def:5.1} and theorem \ref{thm:5.2}, for further discussion.
	
To return to our main task, let us now fix any unstable equilibrium $\mathcal{O}\in \mathcal{A}$ of Morse index $n:=i (\mathcal{O})\geq1$.
It is our task to identify the \emph{predecessors} and \emph{successors}
	\begin{equation}
	w^\iota_\pm:= h_\iota(h^{-1}_\iota (\mathcal{O})\pm1)
	\label{eq:1.15}
	\end{equation}
of $\mathcal{O}$, along the boundary orders $h_\iota$ at $x=\iota=0,1$.
As input information, we will only use the geometric information encoded in the signed Thom-Smale complex $\mathcal{C}^s=\mathcal{C}_f^s$, of Sturm type.
Specifically, the hemisphere refinements are given by the hemisphere decompositions $\Sigma_\pm^j(v)$ of the Sturm complex.
We will illustrate all this for the example of fig.~\ref{fig:1.0} at the end of the present section.
Suffice it here to say that it remains a highly nontrivial task to pass from the partial order, defined by the collection of all signed hemisphere decompositions, to the total order required by $h_0$ and $h_1$\,.

Already \eqref{eq:1.7b} implies that the $\iota$-neighbors $w^\iota_\pm$ of $\mathcal{O}$ possess Morse indices $i(w_\pm^\iota)$ adjacent to $i(\mathcal{O})$:
	\begin{equation}
	\begin{aligned}
	&i(h_\iota(1)) = i(h_\iota(N))=0\,;\\
	&i(w^\iota_\pm)= i(\mathcal{O})\pm(-1)^{i(\mathcal{O})}\mathrm{sign}\,(h^{-1}_{1-\iota}(w^\iota_\pm)-h^{-1}_{1-\iota}(\mathcal{O}))\,.\\
	\end{aligned}
	\label{eq:1.16}
	\end{equation}
For $\iota=0$, this follows from \eqref{eq:1.7b} by elementary substitutions.
For $\iota=1$, this follows from the case $\iota=0$ by the substitution $x\mapsto 1-x$; see also the trivial equivalences \eqref{eq:1.27} below.

To determine the $\iota$-neighbors $w^\iota_\pm$ of $\mathcal{O}$ geometrically, in case $i(w^\iota_\pm)=i(\mathcal{O})-1$, we develop the notion of descendants next. See \cite{firo3d-2} for the special case $n=3$.

\begin{defi}\label{def:1.1}
For fixed $n:=i(\mathcal{O})>0$, let $\mathbf{s}= s_{n-1} \ldots s_0$ denote any sequence of $n$ symbols $s_j\in\{\pm\}$.
Let
	\begin{equation}
	v^j(\mathbf{s})\in \mathcal{E}^j_{s_j}(\mathcal{O})\subseteq \Sigma^j_{s_j}(\mathcal{O}),
	\label{eq:1.18}
	\end{equation}
be defined, inductively for increasing $j=1,\ldots,n-1$, as the unique equilibrium in the signed hemisphere $\Sigma^j_{s_j}(\mathcal{O})$ such that
	\begin{equation}
	v^{j-1}(\mathbf{s}) \in \partial W^u(v^j(\mathbf{s}))=\partial c_{v^j(\mathbf{s})} \, .
	\label{eq:1.19a}
	\end{equation}
For $j=0$ we start the induction with the unique polar equilibria
	\begin{equation}
	\{v^0(\mathbf{s})\}:= \Sigma^0_{s_0} (\mathcal{O})
	\label{eq:1.17}
	\end{equation}
at the two endpoints of the one-dimensional fastest unstable manifold $W^0(\mathcal{O})$.
We call the sequence $v^j(\mathbf{s}),\ j=n-1, \ldots, 0,$ the $\mathbf{s}$-\emph{descendants} of $\mathcal{O}$.
For the constant sequence $\mathbf{s}=++\ldots$\ , we call $v^j(++ \ldots)$ the $+$\emph{descendants} of $\mathcal{O}$.
Similarly, $-$\emph{descendants} have constant $\mathbf{s}=--\ldots$\ .
\emph{Alternating descendants} have alternating sign sequences $s_j$\,.
\end{defi}

For any given $\mathcal{O}$, the descendant $v^j(\mathbf{s})$ only depends on $s_j\,, \ldots, s_0$\,.
In section 2 we show that the descendants $v^j(\mathbf{s})$ are indeed defined uniquely.
We also determine the Morse indices $i(v^j(\mathbf{s}))=j$ and show that the descendants define a \emph{staircase} sequence of heteroclinic orbits between equilibria of descending adjacent Morse indices:
	\begin{equation}
	\mathcal{O}\leadsto v^{n-1}(\mathbf{s})\leadsto \ldots\leadsto v^0(\mathbf{s});
	\label{eq:1.19b}
	\end{equation}
see \eqref{eq:2.3}--\eqref{eq:2.8}.

Clearly, the notion \eqref{eq:1.18} -- \eqref{eq:1.17} of descendants is purely geometric: it is based on the signed hemisphere decomposition $\Sigma^j_\pm(\mathcal{O})=\partial c_{\mathcal{O},\pm}^j$ in the abstract signed hemisphere complex $\mathcal{C}^s$, Beyond the partial order implicit in the hemisphere signs, we do not require any further explicit data on the total boundary orders $h_\iota$\, which we derive.
In section \ref{sec5} we will illustrate this viewpoint, based on specific examples.

In our main result, we will only be concerned with alternating and constant symbol sequences $s_j$\,.
We therefore abbreviate these sequences as follows
	\begin{equation}
	\begin{aligned}
	&\mathbf{s}=+-\ldots:\qquad s_j:=(-1)^{n-1-j} \, ;\\
	&\mathbf{s}=-+\ldots:\qquad s_j:=(-1)^{n-j} \, ;\\
	&\mathbf{s}=++\ldots:\qquad s_j:=+ \, ;\\
	&\mathbf{s}=--\ldots:\qquad s_j:=- \, .\\
	\end{aligned}
	\label{eq:1.20}
	\end{equation}
With this notation we can finally formulate that main result.

\begin{thm} \label{thm:1.2}
Consider any unstable equilibrium $\mathcal{O}$ with  unstable dimension $n= i(\mathcal{O})\geq1$.
Assume that any one of the boundary successors $w^\iota_+$ or predecessors $w^\iota_-$ of $\mathcal{O}$ at $x=\iota=0,1$, as defined in \eqref{eq:1.15}, is more stable than $\mathcal{O}$, i.e.
	\begin{equation}
	i(w^\iota_\pm)=n-1.
	\label{eq:1.21}
	\end{equation}
	
Then $w^\iota_\pm$ of $\mathcal{O}$ is given by the leading descendant $v^{n-1}(\mathbf{s})$ of $\mathcal{O}$, according to the following list:
	\begin{eqnarray}
	w^0_-&=& v^{n-1}(-+\ldots); \label{eq:1.22}\\
	w^0_+&=& v^{n-1}(+-\ldots); \label{eq:1.23}\\
	w^1_-&=& 
	   \begin{cases}
	   v^{n-1}(++\ldots), \qquad \mathrm{for\ even\ n;}\\
       v^{n-1}(--\ldots), \qquad \mathrm{for\ odd\ n;}
	   \end{cases} \label{eq:1.24}\\
	w^1_+&=& 
	   \begin{cases}
	   v^{n-1}(--\ldots), \qquad \mathrm{for\ even\ n;}\\
       v^{n-1}(++\ldots), \qquad\mathrm{for\ odd\ n.}\\
	   \end{cases} \label{eq:1.25}
	\end{eqnarray}
\end{thm}
The above result determines the total boundary orders $h_\iota$ of all equilibria at the two boundaries $x=\iota=0,1$ uniquely.
Indeed, any two equilibria $v_1$ and $v_2$ which are adjacent in the order $h_\iota$\,, say at $\iota=x=0$, possess adjacent Morse indices $i(v_2)=i(v_1)\pm1$ by \eqref{eq:1.16}.
The more unstable equilibrium then qualifies as $\mathcal{O}$, in theorem 1.2, and the other equilibrium qualifies as the predecessor $w^0_-$\,, or as the successor $w^0_+$\,, of $\mathcal{O}$.
Let us therefore start at the top level barycenters of maximal cell dimension $\dim c_\mathcal{O}=i(\mathcal{O})=\dim\mathcal{A}$, but without any a priori knowledge of $h_\iota$ or $\sigma$ in \eqref{eq:1.6},  \eqref{eq:1.7a}.
By \eqref{eq:1.21}, all $h^\iota_-$ neighbors $w^\iota_\pm$ of such $\mathcal{O}$ can then be identified, purely geometrically, as top descendants $v^{n-1}(\mathbf{s})$ of $\mathcal{O}$ with Morse index $i(w_{\pm}^\iota)=\dim\mathcal{A}-1$.
Next, consider all barycenters of cell dimension $\dim \mathcal{A}-1$.
Unless their $h_\iota$-neighbors possess higher Morse index, and those adjacencies have already been taken care of, we may apply theorem 1.2 again to determine their remaining $h_\iota$-neighbors, this time of Morse index $i=\dim \mathcal{A}-2$.
Iterating this procedure we eventually determine all $h_\iota$-adjacencies and our main task is complete.

In sections \ref{sec2} and \ref{sec3} we prove Theorem \ref{thm:1.2} by the following strategy. First we reduce the four cases \eqref{eq:1.22}--\eqref{eq:1.25} to the single case
	\begin{equation}
	\mathbf{s}=++\ldots
	\label{eq:1.26}
	\end{equation}
by four \emph {trivial equivalences}. Indeed, the class of Sturm attractors $\mathcal{A}$ remains invariant under the transformations
	\begin{equation}
	x\mapsto1-x, \quad u\mapsto-u,
	\label{eq:1.27}
	\end{equation}
separately. 
Since the two involutions \eqref{eq:1.27} commute, they generate the Klein 4-group $\mathbb{Z}_2\times\mathbb{Z}_2$ of trivial equivalences.
Since this group acts transitively on the four constant and alternating symbol sequences  \eqref{eq:1.20}, as considered in theorem \ref{thm:1.2}, it is sufficient to consider the case $\mathbf{s}=++\ldots$ of \eqref{eq:1.26}. The remaining cases of \eqref{eq:1.22}--\eqref{eq:1.25} then follow by application of the trivial equivalences. For even $n$, for example, \eqref{eq:1.24} maps to \eqref{eq:1.22} under $x\mapsto 1-x$, to \eqref{eq:1.25} under $u\mapsto -u$, and to \eqref{eq:1.23} under the combination of both.
We henceforth restrict to the case $\mathbf{s}=(++\ldots)$ of \eqref{eq:1.25}. 

In particular, theorem \ref{thm:1.2} solves a long-standing problem.
We already mentioned examples from \cite{fi94, firo96} of Sturm permutations $\sigma, \tau$ which produce the same Thom-Smale complex $\mathcal{C}$ even though $\sigma, \tau$ are not trivially equivalent.
Specifically these are the planar Sturm attractors of the pairs $\sigma= ($2 4 6 8)(3 5 7), $\tau= ($2 6)(3 7)(4 8) or $\sigma= ($2 4 8)(3 5 7), $\tau= ($2 6 4 8)(3 7), in cycle notation with $N=9$ equilibria.
Theorem \ref{thm:1.2} now asserts that the signed Thom-Smale complexes $\mathcal{C}^s$ determine their associated Sturm permutation uniquely.
In general, it is the precise equilibrium targets of the \emph{fast} unstable manifolds which distinguish the Sturm global attractors of $\sigma$ and $\tau$, via the resulting signatures.
In the particular planar examples, these are the 1-dimensional fast unstable manifolds in the 2-cells $c_v$ of equilibria with Morse index $i(v)=2$.

In section 2 we study the descendants of $\mathcal{O}$ for $\mathbf{s}=++\ldots$\ . 
We abbreviate
	\begin{equation}
	v^j:= v^j(++\ldots),
	\label{eq:1.28}
	\end{equation}
for $0\leq j < n= i(\mathcal{O})$.
In sections 3 and 4 we study the additional elements
	\begin{eqnarray}
	\ &\underline{v}^k:&\textrm{the equilibrium } v \in \mathcal{E}^k_+(\mathcal{O})\subseteq \Sigma^k_+(\mathcal{O})
	\textrm { which is closest to } \mathcal{O}, \textrm{ at }x =1, \label{eq:1.29}\\
	\ &\bar{v}^k:&\textrm{the equilibrium } v\in \mathcal{E}^k_+(\mathcal{O})
	\textrm { which is most distant from } \mathcal{O}, \textrm{ at }x =0. \label{eq:1.30}
	\end{eqnarray}
In theorem \ref{thm:3.1} we show $\underline{v}^k = v^k$, for all $0 \leq k<n= i(\mathcal{O})$.
As a corollary, for $k=n-1$, this proves theorem (\ref{thm:1.2}) and completes our task.
In section 4 we show, in addition to $\underline{v}^k=v^k$, that $\underline{v}^k=\bar{v}^k$ for all $k$; see theorem \ref{thm:4.3}.
For an alternative proof of the minimax theorem \ref{thm:4.3} see also the companion paper \cite{rofi20}.
That more elementary proof is based on a more direct and quite detailed ODE analysis of the Sturm meander.
Strictly speaking, theorem \ref{thm:4.3} is not required for the identification task to derive the $h_\iota$-neighbors $w^\iota_\pm$ from $\mathcal{O}$.
However, it much facilitates the task to identify the equilibria $\mathcal{E}^j_\pm(\mathcal{O})$ in the hemispheres $\Sigma^j_\pm(\mathcal{O})$ from the Sturm meander $\mathcal{M}$, in examples.

For a first example let us return to the Thom-Smale 3-ball complex of fig.~\ref{fig:1.0}(b), where
$\mathcal{O}=4,\ i(\mathcal{O})=3,\ w^\iota_- = 3,\ w^0_+=5,\ w^1_+=7$.
See \eqref{eq:1.14} for identification of the hemispheres $\Sigma_\pm^j(\mathcal{O})$.
The equilibria $v\in \mathcal{E}^j_{s_j}(\mathcal{O})$ in the hemispheres $\Sigma^j_{s_j}(\mathcal{O})$, for $s = \pm$, are collected in the table

\begin{minipage}{0.1\textwidth}
\begin{equation}
\label{eq:1.36}
\end{equation}
\end{minipage}
\begin{minipage}{0.85\textwidth}
\begin{center}
\begin{tabular}{|c||c|c|}
\hline
$s_j$  &	$-$	&	$+$
\\ \hline \hline
$j=0$	&	$1= \mathbf{N}$	&	$9=\mathbf{S}$  \\ \hline
$j=1$	&	$2 \in \mathbf{EW}$	&	$8 \in \mathbf{WE}$ 
\\ \hline
$j=2$	&	$3 \in \mathbf{W}$	&	$5, 6, 7 \in \mathbf{E}$
\\ \hline
\end{tabular}
\end{center}
\end{minipage}

Therefore the alternating and constant descendants $v^j(\mathbf{s})$ of $\mathcal{O}= 4$ are given by

\par\smallskip
\begin{minipage}{0.1\textwidth}
\begin{equation}
\label{eq:1.37}
\end{equation}
\end{minipage}
\begin{minipage}{0.85\textwidth}
\begin{center}
\begin{tabular}{|c||c|c|c|c|}
\hline
$\mathbf{s}= s_2s_1s_0$	&	$-+-$	&	$+-+$ &$---$ & $+++$
\\ \hline \hline
$j=0$	&	$1= \mathbf{N}$	&	$9= \mathbf{S}$ &$1= \mathbf{N}$ & $9= \mathbf{S}$ \\ \hline
$j=1$	&	$8 \in \mathbf{WE}$	&	$2 \in \mathbf{EW}$ &$2 \in \mathbf{EW}$ &$8 \in \mathbf{WE}$
\\ \hline
$j=2$	&	$3 = w^0_-\in \mathbf{W}$	&	$5= w^0_+ \in \mathbf{E}$ & $3=w^1_-\in \mathbf{W}$ & $7= w^1_+\in \mathbf{E}$
\\ \hline
\end{tabular}
\end{center}
\end{minipage}
\par\smallskip

The $+$descendants $v^j=v^j(+ + +)$ of $\mathcal{O}=4$, for example, are constructed as $v^0=9=\mathbf{S}$ because $\Sigma^0_+=\mathcal{E}^0_+=\{\mathbf{S}\}$, and $v^1=8$ because $\mathbf{WE}=\Sigma^1_+\supseteq \mathcal{E}^1_+=\{8\}$.
Finally, $v^2\in \mathcal{E}^2_+= \{5,6,7\}\subseteq \Sigma^2_+= \mathbf{E}$ must satisfy $8=v^1\in c_{v^1}\subseteq \partial c_{v^2}$, by recursion \eqref{eq:1.18}, \eqref{eq:1.19a}, and therefore $v^2=7$.
Since $\underline{v}^2= v^2$, by theorem 3.1, we conclude that the successor $w^1_+$ of $\mathcal{O}$, with odd Morse index $n=i(\mathcal{O})=3$, is given by $w^1_+=\underline{v}^2= v^2= v^2(+++)=7$.
This agrees with the equilibrium profiles in fig \ref{fig:1.0}(d).
Note that $\underline{v}^2=7\in \mathbf{E}= \Sigma^2_+$ is in fact the $\mathcal{O}$-closest equilibrium in $\mathcal{E}^2_+=\{5,6,7\}\subseteq \Sigma^2_+$\,, at the right boundary $x=1$.
At the same time, $7=\underline{v}^2=\bar{v}^2$ is also the maximal equilibrium in $\mathcal{E}^2_+=\{5,6,7\}$ above $\mathcal{O}$, at the left boundary $x=0$.
Indeed, the red meander $\mathcal{M}$, in fig.~\ref{fig:1.0}(c), therefore traverses all equilibria  $5,6,7$ in the hemisphere $\mathbf{E}= \Sigma^2_+$\,, after $\mathcal{O}$, with 7 last, before it leaves that open hemisphere forever. 
See also fig.~\ref{fig:1.0}(b): the red meander path $h_0$ of the ordering at $x=0$ leaves the open hemisphere $\mathbf{E}= \Sigma^2_+$ at 7, where the blue path $h_1$ of the ordering at $x=1$ enters the same open hemisphere.
Similarly, the blue path $h_1$ leaves $\mathbf{E}$ at 5, where $h_0$ enters. 
This illustrates theorem \ref{thm:4.3}.
For many more examples see the discussion in section 5, most of which may well be digestible and instructive even before reading the other sections.

The companion paper \cite{rofi20} presents a meander based proof of theorem \ref{thm:4.3}. 
The property $v^{n-1}=\underline{v}^{n-1}$ of theorem \ref{thm:3.1}, which holds independently of theorem \ref{thm:4.3}, then allows us to identify, conversely, the geometric location of predecessors, successors, and signed hemispheres in the associated Thom-Smale complex. 
These results combined, can therefore be viewed as first steps towards the still elusive  goal of a complete geometric characterization of the signed Thom-Smale complexes which arise as Sturm global attractors.

\textbf{Acknowledgments.}
Dear Geneviève Raugel has gently accompanied our long and meandric explorations of Sturm global attractors with kind encouragement, deep understanding, and lasting friendship.
Extended mutually delightful hospitality by the authors is also gratefully acknowledged.
In addition, Clodoaldo Grotta-Ragazzo, Sergio Oliva, and Waldyr Oliva provided an
inspiring and cheerful 24/7 environment at IME-USP: viva!
Very insightful brief comments by the referee were a pleasure to take into account, extensively.
Anna~Karnauhova has contributed the illustrations with her inimitable artistic touch.
Original typesetting was patiently accomplished by Patricia H\u{a}b\u{a}\c{s}escu. 
This work was partially supported by DFG/Germany through SFB 910 project A4, and by FCT/Portugal through projects UID/MAT/04459/2013 and UID/MAT/04459/2019.


\section{Descendants}
\label{sec2}
In this section we fix any unstable hyperbolic equilibrium $\mathcal{O}$ of positive Morse index $n:=i(\mathcal{O})>0$, in a Sturm global attractor.
Let
	\begin{equation}
	\Sigma^{n-1}= \partial W^u (\mathcal{O})= \bigcup\limits_{0\leq j< n}^\centerdot \Sigma^j_\pm
	\label{eq:2.1}
	\end{equation}
be the disjoint signed decomposition of the $(n-1)$-sphere boundary of the $n$-dimensional unstable manifold $W^u(\mathcal{O})$, i.e. we abbreviate $\Sigma^j_\pm := \Sigma^j_\pm(\mathcal{O})$.
Let $\mathcal{E}^j_\pm := \mathcal{E}^j_\pm (\mathcal{O})$ abbreviate the equilibria in hemisphere $\Sigma^j_\pm \subseteq \partial W^{j+1}(\mathcal{O})$.
From \eqref{eq:1.10} we recall
	\begin{equation}
	\mathcal{E}^j_\pm = \{v\in \mathcal{E}\mid z(v- \mathcal{O})= j_\pm \textrm{ and } \mathcal{O}\leadsto v \}.
	\label{eq:2.2}
	\end{equation}
Concerning the descendants $v^j=v^j(\mathbf{s})$ of $\mathcal{O}$, according to definition \ref{def:1.1}, we also fix any sequence $\mathbf{s}=s_{n-1}\ldots s_0$ of $n$ signs $s_j = \pm$, for $0\leq j < n$.
We first explain why the descendants $v^j$ are well-defined. After a pigeon hole proposition \ref{eq:2.1}, we  collect some elementary properties of descendants in lemma \ref{lem:2.2}.

Except for that last lemma, we do not require the sign sequence $\mathbf{s}$ to be constant or alternating.
We do not require assumption \eqref{eq:1.21} of theorem \ref{thm:1.2} to hold anywhere, in the present section.
In particular, the descendant $v^{n-1}$ here need not coincide with any immediate successor or predecessor $w_\pm^\iota$ of $\mathcal{O}$ on any boundary $x=\iota=0,1$.

Let us examine the recursive definition \ref{def:1.1} first. For $s_0= \pm$, the equilibrium $\{v^0\} :=\Sigma^0_{s_0}(\mathcal{O})$ is defined uniquely by \eqref{eq:1.17}.
Now consider $1 \leq j <n$ and assume $v^0, \ldots, v^{j-1}$ have been well-defined, already.
By the Schoenflies result \cite{firo13} on the $j$-sphere boundary $\Sigma^j= \partial W^{j+1}$ of the $(j+1)$-dimensional fast unstable manifold $W^{j+1}=W^{j+1}(\mathcal{O})$ of $\mathcal{O}$, we have the disjoint decomposition $\text {clos } \Sigma^j_{s_j} = \Sigma^j_{s_j} \dot{\cup} \Sigma^{j-1} $, and hence
	\begin{equation}
	v^{j-1}\in \Sigma^{j-1}_{s_{j-1}} \subseteq \Sigma^{j-1}= \partial \Sigma^j_{s_j} \,.
	\label{eq:2.3}
	\end{equation}
	
We claim that there exists a unique cell $c_{v^j}=W^u(v^j)$ in
	\begin{equation}
	\Sigma^j_{s_j} = \bigcup \limits_{v \in \mathcal{E}^j_{s_j}}  W^u(v),
	\label{eq:2.4}
	\end{equation}
such that \eqref{eq:1.19a} holds, i.e. such that 
	\begin{equation}
	v^{j-1}\in \partial c_{v^j}\, .
	\label{eq:2.5}
	\end{equation}
This follows again from \cite{firo13}, which asserts the following.
Let $\varphi_j$ denote the $j$-th Sturm-Liouville eigenfunction of the linearization at $\mathcal{O}$, with sign chosen such that $z(\varphi_j)=j_+$\,.
The eigenprojection $P^j$ projects the closed $j$-dimensional hemisphere $\mathrm{clos}\, \Sigma^j_{s_j}$ into the tangent space $T_\mathcal{O}W^j= \mathrm{span} \{ \varphi_0,\ldots,\varphi_{j-1} \big \}$ at $\mathcal{O}$ of the $j$-dimensional fast unstable manifold $W^j(\mathcal{O})$.
The projection $P^j$ is homeomorphic onto a topological $j$-dimensional ball with Schoenflies $(j-1)$-sphere boundary.
This homeomorphic projection preserves the regular Thom-Smale cell decomposition of clos\ $\Sigma^j_{s_j}$\,.
In particular, any $(j-1)$-cell in the $j$-dimensional interior hemisphere $\Sigma^j_{s_j}$ possesses precisely two $j$-cell neighbors in $\Sigma^j_{s_j}$\,, separating them as a shared boundary. 
The $(j-1)$-cell $c_{v^{j-1}}\subseteq \Sigma ^ {j-1}= \partial \Sigma^j_{s_j}$ in the $j-1$-dimensional boundary, however, possesses a \emph{ unique} $j$-cell neighbor $c_{v^j}\in \Sigma^j_{s_j}$ such that 
	\begin{equation}
	c_{v^{j-1}}\subseteq \partial c_{v^j}\,.
	\label{eq:2.6}
	\end{equation}
Indeed, $c_{v^j}$ must be a $j$-cell, recursively in $j$, because \eqref{eq:2.6} implies $j-1=\dim\, c_{v^{j-1}}<\dim\, c_{v^j}\leq \dim\, \Sigma_{s_j}^j = j$\,.
This proves that \eqref{eq:2.5} defines $v^j$ uniquely, and explains why all descendants $v^j$ are well-defined, in definition \ref{def:1.1}.

Since $c_{v^j}= W^u(v^j)$ is a $j$-cell, in our construction of descendants, we immediately obtain the Morse indices
	\begin{equation}
	i(v^j)=j,
	\label{eq:2.7}
	\end{equation}
for all $0\leq j<n$.
As we have mentioned in \eqref{eq:1.5b} already, \eqref{eq:2.6} alias $v^{j-1}\in \partial W^u(v^j)$ implies $v^j\leadsto v^{j-1}$,
by \cite{he85} and the Sturm transversality property.
This proves the unique staircase  \eqref{eq:1.19b} of heteroclinic orbits, i.e.
	\begin{equation}
	\mathcal{O}\leadsto v^{n-1}\leadsto \ldots\leadsto v^1 \leadsto v^0,
	\label{eq:2.8}
	\end{equation}
with $v^j \in \mathcal{E}_{s_j}^j(\mathcal{O})$.
This heteroclinic staircase with Morse indices descending by 1, stepwise, motivates the name "descendants" for the equilibria $v^j$.
Note that Sturm transversality of stable and unstable manifolds implies transitivity of the relation "$\leadsto$".
In particular, not only does $\mathcal{O}$ connect to any $v^j\in \Sigma^j_{s_j}(\mathcal{O})$, but also
	\begin{equation}
	0\leq j < k < n \quad \Rightarrow \quad v^k \leadsto v^j.
	\label{eq:2.9}
	\end{equation}
Any heteroclinic orbit $v^j\leadsto v^{j-1}$ in the staircase of descendants, from Morse index $j$ to adjacent Morse index $j-1$, is also known to be unique; see {\cite[Lemma 3.5]{brfi89}}.

We briefly sketch an alternative possibility to construct the heteroclinic staircase \eqref{eq:2.8}, directly.
Our construction is based on the $y$-map, first constructed in \cite{brfi88} by a topological argument. 
The $y$-map allows us to identify at least one solution $u(t,x)$, with initial condition $u(0, \cdot)$ in any small sphere around $\mathcal{O}$ in $W^u(\mathcal{O})$, such that the signed zero numbers
\begin{equation}
	z(u(t, \cdot)- \mathcal{O})= j_{s_j}
	\label{eq:2.9a}
	\end{equation}
are prescribed for $t_j<t<t_{j-1},\ 0\leq j < n= i (\mathcal{O})$.
Here $t_{n-1}:= -\infty,\ t_{-1}:= +\infty$, and the remaining dropping times $t=t_j$ of the zero number $z$ can be chosen arbitrarily.
Consider sequences of $t_j$ such that the length of each finite interval $(t_j, t_{j-1})$ tends to infinity.
Separately for each $1\leq j<n$, we consider the resulting time-shifted sequences $u^j(t,\cdot) := u(t-t_{j-1}\,,\cdot)$.
Passing to locally uniformly convergent subsequences, the $u^j$ will then converge to the desired heteroclinic orbits
	\begin{equation}
	u^j(t, \cdot): \quad \mathcal{E}^j_{s_j}\ni v^j\leadsto v^{j-1} \in \mathcal{E}^{j-1}_{s_{j-1}}\,,
	\label{eq:2.9b}
	\end{equation}
for $1\leq j< n$.
Here dropping of Morse indices along any heteroclinic orbit $v\leadsto w$ implies that the $v^j$ constructed from adjacent $t$-intervals in fact coincide.
Moreover $i(v^j)=j$.
By construction of $u(0, \cdot)\in W^u (\mathcal{O})$, we also have $\mathcal{O}\leadsto v^{n-1}$.
Uniqueness of the heteroclinic staircase, however, is not obtained by the above topological construction.

Before we collect more specific properties of $+$descendants, in lemma \ref{lem:2.2}, we record a useful pigeon hole triviality which we invoke repeatedly below.

\begin{prop}\label{prop:2.1}
Let $\zeta_j\geq 0$ be a strictly increasing sequence of m integers, $0\leq j<m$, which satisfy
	\begin{equation}
	\zeta_m<m\,.
	\label{eq:2.10}
	\end{equation}
Then
	\begin{equation}
	\zeta_j=j, \qquad \mathrm{for\ all}\ j\,.
	\label{eq:2.11}
	\end{equation}
\end {prop}
For example, the heteroclinic staircase \eqref{eq:2.8}, with $\zeta_j:=i(v^j)\geq 0$ and $m:=n=i(\mathcal{O})$, reaffirms $i(v^j)=\zeta_j=j$, as already stated in \eqref{eq:2.7}.

In the following we call $v^j$ with $j$ even the \emph{even descendants. Odd descendants} $v^j$ refer to odd $j$.
We occasionally use the abbreviations
	\begin{equation}
	v_1<_0 v_2, \ \mathrm{and}\  v_1 < _1 v_2,
	\label{eq:2.12}
	\end{equation}
to indicate that $v_1(x)< v_2(x)$ holds at $x=0$, and at $x=1$, respectively.

\begin{lem}\label{lem:2.2}
Consider the $+$descendants $v^j=v^j(\mathbf{s})$ of $\mathcal{O}$, i.e.~with constant sequence $\mathbf{s}= ++\ldots$~.
Then the following statements hold true for any $0\leq j, k<n=i(\mathcal{O})$:
	\begin{itemize}
	\item[(i)] $j<k \ \Rightarrow\ v^j>_0 v^k,  \quad at~ x=0$;
	\item[(ii)] $\mathcal{O}<_0 v^{n-1}<_0 \ldots<_0 v^0, \quad at~ x=0$;
	\item[(iii)] for even k and even descendants, 
		$\mathcal{O}<_1v^k<_1\ldots<_1v^2<_1v^0, \quad at~ x=1$;
	\item[(iv)] for odd k and odd descendants, 
		$\mathcal{O}>_1 v^k>_1\ldots>_1v^3>_1v^1, \quad at~ x=1$;
	\item[(v)]$z(v^j-v^k)=\min\{j,k\}, \quad for ~ j\neq k$;
	\item[(vi)] $+$descendancy is transitive, i.e.~$+$descendants, of $+$descendants $v^k$ of $\mathcal{O}$, are $+$descendants of $			\mathcal{O}$.
	\end{itemize}
\end{lem}

\begin{proof}[\textbf{Proof:}]
To prove (i), indirectly, suppose $v^j<_0 v^k$.
For the $+$descendant $v^j\in \Sigma^j_+(\mathcal{O})$ we have $\mathcal{O} <_0 v^j$.
Therefore $v^j$ is between $\mathcal{O}$ and $v^k$, at $x=0$.
For the heteroclinic orbit $u(t, \cdot)$ from $\mathcal{O}$ to $v^k$ this implies the strict dropping
	\begin{equation}
	j=z(\mathcal{O}-v^j)= \lim_{t\rightarrow -\infty} z(u(t, \cdot)-v^j)> \lim_{t\rightarrow+\infty}z(u(t, \cdot)-v^j)= z(v^k-v^j);
	\label{eq:2.13}
	\end{equation}
see \eqref{eq:1.3}, \eqref{eq:1.4}.
Indeed, for $u(t_0, 0)=v^j(0)$, a multiple zero of $x\mapsto u(t_0,x)- v^j(x)$ occurs at the left Neumann boundary $x=0$.

On the other hand, the $z$-inequalities
	\begin{eqnarray}
		u\in W^u (v)\quad &\Rightarrow \quad z(u-v)< i(v) 
	\label{eq:2.14} \\
		u\in W^s (v) \setminus  \{v\}\quad &\Rightarrow \quad z(u-v)\geq i(v) 
	\label{eq:2.15}
	\end{eqnarray}
were already observed in \cite{brfi86}; see also \cite{firo13} for a more recent account.
Hence the heteroclinic orbit $u(t, \cdot)\in W^s(v_j)\setminus\{v_j\}$ from $v^k$ to $v^j$, for $k>j$, induced by the +descending heteroclinic staircase \eqref{eq:2.8}, \eqref{eq:2.9}, implies
	\begin{equation}
	z(v^k-v^j)=\lim_{t\mapsto -\infty} z(u(t, \cdot)-v^j)\geq i(v^j)=j,
	\label{eq:2.16}
	\end{equation}
in view of the Morse indices \eqref{eq:2.7}.
The contradiction between \eqref{eq:2.16} and \eqref{eq:2.13} proves claim (i).

Claim (ii) is an immediate consequence of $\mathcal{O}< _0 v^{n-1}\in \Sigma^{n-1}_+ (\mathcal{O})$ and property (i).

We prove claim (iii) next, where $0\leq j<k$ are both even.
Suppose, indirectly, that $v^j<_1 v^k$.
Since $s_j=+$ indicates $v^j\in \Sigma^j_+(\mathcal{O})$, we have $\mathcal{O}<_0 v^j$ and $z(v^j-\mathcal{O})=j$.
Since $j$ is even, we also have $\mathcal{O}<_1 v^j<_1 v^k$.
As in \eqref{eq:2.13}, strict dropping of $z(u(t,\cdot)-v^j)$ for $u(t,\cdot):\ \mathcal{O}\leadsto v^k$, this time when $u(t_0, 1)=v^j(1)$ at $x=1$, therefore implies
	\begin{equation}
	j=z(\mathcal{O}-v^j)> z(v^k-v^j).
	\label{eq:2.17}
	\end{equation}
As in \eqref{eq:2.15}, \eqref{eq:2.16}, transitive $v^k\leadsto v^j$ on the other hand implies
	\begin{equation}
	z(v^k-v^j)\geq i(v^j)=j.
	\label{eq:2.18}
	\end{equation}
This contradiction proves claim (iii) on even $j,k$.

The case (iv) of odd $j,k$ is analogous.
We just argue indirectly, for odd $j<k$ and $v^k<_1 v^j$, via $\Sigma^j_+(\mathcal{O})\ni v^j<_1\mathcal{O}$.

To prove claim (v), consider $0\leq j<k<n=i(\mathcal{O})$.
To show $0\leq \zeta_j := z(v^j-v^k)=j$, for those $j$, we invoke the pigeon hole proposition \ref{prop:2.1}.
Assumption \eqref{eq:2.10} holds, for $m:=k$, because $v^k\leadsto v^j$ and \eqref{eq:2.14} imply $0\leq z(v^j-v^k)<i(v_k)=k$.
To show that the sequence $\zeta_j$ increases strictly, with $j$, we compare $\zeta_{j-1}$ and $\zeta_j$ for $1\leq j< k$.
Since $j-1$ and $j$ are of opposite even/odd parity, $v^{j-1}$ and $v^j$ lie on opposite sides of $v^k$, at $x=1$; see (iii), (iv).
Therefore $v^j\leadsto v^{j-1}$ implies strict dropping of $z$
	\begin{equation}
	\zeta_j= z(v^j-v^k)> z(v^{j-1}-v^k)=\zeta_{j-1} \,.
	\label{eq:2.19}
	\end{equation}
Hence pigeon hole proposition \ref{prop:2.1} proves claim (v).

It remains to prove claim (vi).
Consider the $+$descendants $v^k\in\Sigma^k_+(\mathcal{O})$.
In view of definition \ref{eq:2.1}, \eqref{eq:1.18}, \eqref{eq:1.19a}, and \eqref{eq:2.2}, the $+$descendants of $\mathcal{O}$ are uniquely characterized by the descendant heteroclinic staircase
	\begin{equation}
	\mathcal{O}\leadsto v^{n-1}\leadsto \ldots \leadsto v^k \leadsto \ldots \leadsto v^j\leadsto \ldots \leadsto v^0
	\label{eq:2.20}
	\end{equation}
together with the conditions
	\begin{equation}
	z(v^j-\mathcal{O})=j_+ \, .
	\label{eq:2.21}
	\end{equation}
To show that the unique $+$descendants $\tilde{v}^j\in\Sigma^j_+(v^k)$ of $v^k$, for $0\leq j<k$, coincide with the $+$descendants $v^j$ of $\mathcal{O}$, it only remains to show
	\begin{equation}
	z(v^j-v^k)= j_+ \, .
	\label{eq:2.22}
	\end{equation}
Property (v) asserts $z(v^j-v^k)=j$, since $0 \leq j < k$.
Ordering (i) asserts $v^j-v^k>_00$.
This proves \eqref{eq:2.22}, claim (vi), and the lemma.
\end{proof}

We conclude this section with an illustration of the action, on lemma 2.2 (ii)--(iv), of the four trivial equivalences generated by $u\mapsto -u$ and $x\mapsto 1-x$ from \eqref{eq:1.27}. 
First note that (ii) refers to a \emph{monotone order} of all descendants $v^k$, whereas (iii) and (iv) address the \emph{alternating order}, depending on the even/odd parity of $k$, at the opposite end of the $x$-interval.
The trivial equivalence $u\mapsto -u$ flips the monotone order (ii) into the opposite monotone order $\mathcal{O}>v^{n-1}>\ldots>v^0$, at $x=0$, which corresponds to constant  $s_j=-$.
The trivial equivalence $x\mapsto 1-x$, in contrast, makes the monotone order (ii) and its opposite appear at $x=1$, respectively.
Therefore, the four trivial equivalences are characterized by the unique one of the four half axes of $u$, at $x=0$ and $x=1$, where all descendants are ordered monotonically. 
The alternating orders appear on the $x$-opposite $u$-axis.
In summary, $x \mapsto 1-x$ swaps constant with alternating sign sequences $\mathbf{s}$,
whereas $u \mapsto -u$ reverses the sign of $\mathbf{s}$.


\section{First descendants and nearest neighbors}\label{sec3}

In this section we prove our main result, theorem \ref{thm:1.2}. 
As explained in the introduction, the trivial equivalences \eqref{eq:1.27} reduce the four cases \eqref{eq:1.22}--\eqref{eq:1.25} to the single case $\mathbf{s}=++\ldots$ of $+$descendants $v^k=v^k(++\ldots)$ with $k=n-1,\  n:=i(\mathcal{O})$; see \eqref{eq:1.26}, \eqref{eq:1.28}.
We also recall the notation $v=\underline{v}^k$ of \eqref{eq:1.29} for the equilibrium $v\in \mathcal{E}^k_+(\mathcal{O})\subseteq \Sigma^k_+(\mathcal{O})$ which is closest to $\mathcal{O}$ at $x=1$.
In theorem \ref{thm:3.1} below, we show $\underline{v}^k= v^k$, for all $0\leq k<n$.
Invoking theorem \ref{thm:3.1} for the special case $k=n-1$, we will then prove theorem \ref{thm:1.2}.
\begin{thm} \label{thm:3.1}
With the above notation, and in the setting of the introduction,
	\begin{equation}
	\underline{v}^k=v^k
	\label{eq:3.1}
	\end{equation}
holds for all $0\leq k< n = i(\mathcal{O})$.	

This result holds true, independently of the particular Morse indices $i(w_\pm^\iota)$ of the immediate $h_\iota$-neighbors $i(w_\pm^\iota)$ of $\mathcal{O}$. 
\end{thm}

\begin{proof}[\textbf{Proof:}]
Let  $v$ denote any equilibrium in $\mathcal{E}_+^k(\mathcal{O})$, i.e. $z(v-\mathcal{O})=k_+$\,.
At $x=1$ this implies $(-1)^k(v - \mathcal{O}) >_1  0$.
In other words, all equilibria $v\in \mathcal{E}_+^k(\mathcal{O})$ are on the same side of $\mathcal{O}$, at $x=1$.

To prove \eqref{eq:3.1} indirectly, suppose $\underline{v}^k \neq v^k$.
Since $\underline{v}^k, v^k\in \mathcal{E}^k_+(\mathcal{O})$, the definition \eqref{eq:1.29} of $\underline{v}^k$ implies that $\underline{v}^k$ is strictly closer to $\mathcal{O}$ than $v^k$, at $x=1$:
	\begin{equation}
	\mathcal{O}<_1(-1)^k \underline{v}^k<_1 (-1)^k v^k.
	\label{eq:3.2}
	\end{equation}
Consider any $1\leq j\leq k<n$.
Then the part
	\begin{equation}
	v^k\leadsto \ldots \leadsto v^j \leadsto v^{j-1}\leadsto \ldots \leadsto v^0
	\label{eq:3.3}
	\end{equation}
of the descending heteroclinic staircase \eqref{eq:1.19b} of the $\mathcal{O}$ descendants implies $v^j\leadsto v^{j-1}$.
Lemma \ref{lem:2.2}(iii), (iv) and \eqref{eq:3.2} imply that $\underline{v}^k$ is strictly between $v^j$ and $v^{j-1}$ at $x=1$, due to the opposite parity of $j-1$ and $j$, mod 2.
Strict dropping of the zero number $t \mapsto z(u(t,\cdot)-\underline{v}_k)$ along the heteroclinic orbit $u(t,\cdot)$ from $v^j$ to $v^{j-1}$, at the boundary  $x=1$, therefore implies
	\begin{equation}
	\zeta_j:=z(v^j-\underline{v}^k)>z(v^{j-1}- \underline{v}^k)=\zeta_{j-1}\geq0.
	\label{eq:3.4}
	\end{equation}
We have already used dropping arguments of this type in our proof of lemma \ref{lem:2.2}, repeatedly.
In \cite[Lemma 3.6(i)]{firo13}, on the other hand, we have already observed that
	\begin{equation}
	\label{eq:3.5a}
	z(u_1-u_2) < k
	\end{equation}
for any two elements $u_1\,, u_2$ of the same closed hemisphere $\mathrm{clos}\,\Sigma^k_+ (\mathcal{O})$.
In particular
	\begin{equation}
	\zeta_k = z(v^k-\underline{v}^k)\leq k-1,
	\label{eq:3.5}
	\end{equation}
for the distinct equilibria $v^k, \underline{v}^k \in \Sigma^k_+ (\mathcal{O})$.
By the standard pigeon hole argument of proposition \ref{prop:2.1}, however, the $k+1$ distinct integers $0\leq\zeta_0<\ldots<\zeta_k$ of \eqref{eq:3.4} cannot fit into the $k$ available slots ${0, \ldots, k-1}$ of \eqref{eq:3.5}.
This contradiction proves the theorem.	
\end{proof}	

\begin{proof}[\textbf{Proof of theorem \ref{thm:1.2}:}]
Suppose first that $n=i(\mathcal{O})>0$ is even.
Assume $i(w^1_-)=n-1$ for the $h_1$-predecessor $w^1_-$ of $\mathcal{O}$ at $x=1$, see \eqref{eq:1.21}.
We then have to show assertion \eqref{eq:1.24}, i.e.
	\begin{equation} 
	w^1_-=v^{n-1} (++\ldots)=v^{n-1},
	\label{eq:3.6}
	\end{equation}
in the notation of the present section.
In particular we may consider $+$descendants $\mathbf{s}= ++\ldots$\ .

We first claim $\mathcal{O}\leadsto w^1_-$\,.
To prove that claim we recall Wolfrum's Lemma; see \cite{wo02} and \cite[Appendix]{firo3d-2}: for equilibria $v_1, v_2$ with $i(v_1)>i(v_2)$ we have $v_1\leadsto v_2$ if, and only if, there does not exist any equilibrium $w$ with boundary values strictly between $v_1$ and $v_2$\,, at $x=0$ or $x=1$ alike, such that
	\begin{equation}
	z(v_1-w)=z(w-v_2)=z(v_1-v_2).
	\label{eq:3.7}
	\end{equation}
For $v_1:=\mathcal{O}$ and $v_2:=w^1_-$\,, by definition \eqref{eq:1.15} of $w^1_-$ as the predecessor of $\mathcal{O}$ at $x=1$, there do not exist any equilibria at all between $v_1$ and $v_2$\,, at $x=1$.
Therefore $i(\mathcal{O})=n>n-1=i(w^1_-)$ implies $\mathcal{O}\leadsto w^1_-$\,, as claimed.

We claim $w^1_-\in \mathcal{E}^{n-1}_+(\mathcal{O})$ next.
Since $\mathcal{O}\leadsto w^1_-$ with adjacent Morse indices $n$ and $n-1$, properties \eqref{eq:2.14},\eqref{eq:2.15} of the zero numbers on unstable and stable manifolds imply
	\begin{equation}
	\label{eq:3.8}
	n= i(\mathcal{O})> z(\mathcal{O}-w^1_-)\geq i(w^1_-)= n-1,
	\end{equation}
i.e. $z(\mathcal{O}- w^1_-)=n-1$.
Definition \eqref{eq:1.15} of $w^1_-$ as the predecessor of $\mathcal{O}$ at $x=1$ implies $w^1_- <_1 \mathcal{O}$.
Since $n$ is even and $z(\mathcal{O}- w^1_-)=n-1$ is odd, this implies $w^1_->_0 \mathcal{O}$, at $x=0$.
In view of definition \eqref{eq:1.10}, this proves $w^1_-\in \mathcal{E}^{n-1}_+(\mathcal{O})$.
As predecessor of $\mathcal{O}$, therefore, $w^1_-$ is closest to $\mathcal{O}$ at $x=1$ in $\mathcal{E}^{n-1}_+(\mathcal{O})$.
In other words $w^1_-=\underline{v}^{n-1}$, by definition \eqref{eq:1.29} of $\underline{v}^{n-1}$.
Invoking theorem \ref{thm:3.1} for $k=n-1$ therefore shows
	\begin{equation}
	\label{eq:3.9}
	w^1_-=\underline{v}^{n-1}= v^{n-1},
	\end{equation}
as claimed in \eqref{eq:3.6}, for even $n$.

For odd $n$ and even $n-1$, we can repeat the exact same steps for the successor $w^1_+$ of $\mathcal{O}$ at $x=1$, instead of the predecessor $w^1_-$\,.
This proves theorem \ref{thm:1.2} for $\mathbf{s}=++\ldots$\ .

The remaining cases of constant or alternating $\mathbf{s}$ follow by the trivial equivalences \eqref{eq:1.27}, as already indicated in the introduction.
\end{proof}


\section{Minimax: the range of hemispheres}\label{sec4}

For the $+$descendants $v^k=v^k(++\ldots)$ of $\mathcal{O}$, with $0\leq k<n:= i(\mathcal{O})$, we have shown
	\begin{equation}
	v^k=\underline{v}^k
	\label{eq:4.1}
	\end{equation}
in the previous section.
See theorem \ref{thm:3.1}, where $\underline{v}^k$ denoted the equilibrium closest to $\mathcal{O}$, at $x=1$, in the hemisphere $\Sigma^k_+(\mathcal{O})$.
In theorem \ref{thm:4.3} of the present section we show the alternative characterization
	\begin{equation}
	v^k=\bar{v}^k,
	\label{eq:4.2}
	\end{equation}
where $\bar{v}^k$ denotes the equilibrium most distant from $\mathcal{O}$ at the opposite boundary $x=0$, in the same hemisphere $\Sigma^k_+(\mathcal{O})$.
See definitions \eqref{eq:1.29}, \eqref{eq:1.30} of $\underline{v}^k, \bar{v}^k$.
In particular, the minimax formulation
	\begin{equation}
	\underline{v}^k=v^k=\bar{v}^k,
	\label{eq:4.3}
	\end{equation}
shows how, within $\Sigma_k^+(\mathcal{O})$, minimal distance from $\mathcal{O}$, along the meander axis $h_1$ of $x=1$,  coincides with maximal distance from $\mathcal{O}$, along the meander $h_0$ of $x=0$.

Throughout this section we fix $k$.
In lemma \ref{lem:4.1} we show 
	\begin{equation}
	i(\bar{v}^k)=k,
	\label{eq:4.4}
	\end{equation}
in correspondence to $i({v}^k)=k$.
We then study the $+$descendants $w^j\in \Sigma^j_+(\bar{v}^k)$ of $\bar{v}^k$, for $0\leq j<k$.
In lemma \ref{lem:4.2}, in particular, we show
	\begin{equation}
	z(w^{k-1}- \mathcal{O})=k-1.
	\label{eq:4.5}
	\end{equation}
Combining \eqref{eq:4.4} and \eqref{eq:4.5} will then prove the claim \eqref{eq:4.2} of theorem \ref{thm:4.3}.

We conclude the section, in corollary \ref{cor:4.4}, with a summary of our results for all four cases of constant and alternating descendants.

\begin{lem}\label{lem:4.1}
Claim \eqref{eq:4.4} holds true, i.e., $i(\bar{v}^k)=k$ for any $0\leq k < n = i(\mathcal{O})$.
\end{lem}

\begin{proof}[\textbf{Proof:}]
We first show $i(\bar{v}^k)= \dim W^u(\bar{v}^k)\leq k.$
Indeed
	\begin{equation}
	\bar{v}^k\in \mathcal{E}^k_+(\mathcal{O})\subseteq \Sigma^k_+(\mathcal{O}) = \bigcup \limits_{v\in \mathcal{E}^k_+(\mathcal{O})}W^u(v)
	\label{eq:4.6}
	\end{equation}
implies $W^u(\bar{v}^k)\subseteq \Sigma^k_+(\mathcal{O})$, and hence $\dim W^u(\bar{v}^k) \leq \dim \Sigma^k_+(\mathcal{O}) = k$.

To prove $i(\bar{v}^k)=k$, indirectly, we may therefore suppose $i(\bar{v}^k)<k$.
We will then reach a contradiction in two steps, below.
In step 1 we show that there exists an initial condition $u_0\in \Sigma_+^k$ such that
	\begin{equation}
	\label{eq:4.6a}
	z(u_0-\bar{v}^k) = (k-1)_+\,.
	\end{equation} 
Since $u_0$ is chosen in the invariant open hemisphere $\Sigma_+^k$\,, we can follow the nonlinear backwards trajectory $u(t,\cdot)$ from $u_0$\,, for $t\leq 0$, and define a second equilibrium $\tilde{v}$ as the $\alpha$-limit set of $u(t,\cdot)$.
In step 2 we then show that $\tilde{v}$ satisfies 
	\begin{equation}
	\label{eq:4.6b}
	\tilde{v} \in \Sigma^k_+(\mathcal{O}) \quad \mathrm{and} \quad \tilde{v}>_0 \bar{v}^k >_0 \mathcal{O}.
	\end{equation}
But \eqref{eq:4.6b} contradicts maximality of $\bar{v}^k$ in $\Sigma^k_+(\mathcal{O})$, at $x=0$; see definition \eqref{eq:1.30} of $\bar{v}^k$.
This contradiction will prove the lemma.

\emph{Step 1:} Construction of the initial condition $u_0$ in claim \eqref{eq:4.6a}.

As in the proof of claim \eqref{eq:2.5} above, let $P^k$ denote the eigenprojection onto the tangent space $T_\mathcal{O} W^k = \mathrm{span}\, \{\varphi_0,\ldots,\varphi_{k-1}\}$ at $\mathcal{O}$ of the fast unstable manifold $W^k(\mathcal{O})$.
Based on the Schoenflies arguments in \cite{firo13}, as before, $P^k$ projects the $k$-dimensional closed hemisphere $\mathrm{clos} \, \Sigma^k_+(\mathcal{O})$ homeomorphically into the tangent space $T_\mathcal{O}W^k(\mathcal{O})$.
Since $\bar{v}^k$ is in the interior of the open hemisphere $\Sigma_+^k$\,, we may fix $\varepsilon>0$ small enough and define $u_0\in \Sigma_+^k \setminus \{\bar{v}^k\}$, arbitrarily close to $\bar{v}^k$, by $P^k(u_0-\bar{v}^k) := \varepsilon \varphi_{k-1}$\,, i.e.
	\begin{equation}
	\label{eq:4.6c}
	u_0-\bar{v}^k = \varphi := \varepsilon \varphi_{k-1} + a_k\varphi_k + \ldots \ .
	\end{equation}
We have used local bijectivity of the homeomorphic projection $P^k$ here.
We also fixed the sign of $\varphi_{k-1}$ such that $z(\varphi_{k-1})=(k-1)_+$\,.

Applying \eqref{eq:3.5a} with $u_1=u_0$ and $u_2=\bar{v}^k$, we obtain $z(\varphi)\leq k-1$.
Sturm's observation \eqref{eq:1.5a}, on the other hand, extends to show $z(\varphi)\geq k-1$ in \eqref{eq:4.6c}.
Therefore $z(\varphi)=k-1$.
Upon closer inspection, the same arguments imply $z(\psi(t,\cdot))\equiv k-1$, for the linearized parabolic solution $\psi(t,\cdot)$ in \eqref{eq:SLpar} and all $t\geq0$.
Invoking the limit $t\rightarrow +\infty$, where the positive coefficient of $\varphi_{k-1}$ in $\psi(t,\cdot)$ dominates, this proves the signed zero number
	\begin{equation}
	\label{eq:4.6d}
	z(\psi(t,\cdot)) \equiv z(\varphi_{k-1}) = (k-1)_+\,,\quad\textrm{for all }t\geq0\,.
	\end{equation}
Indeed the constant zero number $z(\psi(t,\cdot))\equiv k-1$ cannot drop at $x=0$, ever.
Inserting $t=0$ in \eqref{eq:4.6d}, and recalling the construction \eqref{eq:4.6c} of $\psi(0,\cdot)=\varphi=u_0-\bar{v}^k$ proves claim \eqref{eq:4.6a}.

\emph{Step 2}: Proof of claim \eqref{eq:4.6b}.

By construction, $\tilde{v} = \lim u(t,\cdot)$ for $t\rightarrow -\infty$.
Then \eqref{eq:3.5a}, with $u_1=u(t,\cdot)$ and $u_2=\bar{v}^k$, implies $z(u(t,\cdot)-\bar{v}^k)\leq k-1$, for all $t\leq 0$.
Zero number dropping \eqref{eq:1.4}, on the other hand, now extends our previous claim \eqref{eq:4.6a} to  $z(u(t,\cdot)-\bar{v}^k)\geq k-1$, for all $t\leq 0$.
Therefore $z(u(t,\cdot)-\bar{v}^k)\equiv k-1$, for all $t\leq 0$.
Since the zero number never drops, \eqref{eq:4.6a} implies the signed zero number
	\begin{equation}
	\label{eq:4.6d1}
	z(u(t,\cdot)-\bar{v}^k) \equiv z(u_0-\bar{v}^k) = (k-1)_+\,,\quad\textrm{for all }t\leq0\,.
	\end{equation}
Passing to the limit $t\rightarrow-\infty$ in \eqref{eq:4.6d1} shows $z(\tilde{v}-\bar{v}^k) = (k-1)_+$\,, and in particular $\tilde{v} >_0 \bar{v}^k$, as claimed in \eqref{eq:4.6b}.

It remains to show $\tilde{v} \not\in \partial \Sigma_+^k(\mathcal{O})= \Sigma^{k-1}(\mathcal{O})$.
Our construction of $\tilde{v}$ as the $\alpha$-limit set of $u(t,\cdot) \in W^u(\tilde{v})$, and \eqref{eq:2.14}, \eqref{eq:3.5a} imply
	\begin{equation}
	\label{eq:4.6e}
	i(\tilde{v})>z(u_0-\tilde{v})=z(\bar{v}^k-\tilde{v})=k-1=\dim\,\Sigma^{k-1}(\mathcal{O}).
	\end{equation}
In the first equality here, we used that $u_0$ can be chosen arbitrarily close to $\bar{v}^k$, for small $\varepsilon>0$.
This contradiction to $\tilde{v} \in W^u(\tilde{v}) \subset \Sigma^{k-1}(\mathcal{O})$ completes the proof of claim \eqref{eq:4.6b}, and of the lemma.
\end{proof}

We consider the $+$descendants $w^j\in \Sigma^j_+(\bar{v}^k)$ of $\bar{v}^k$ next, for $0\leq j< k$:
	\begin{equation}
	\bar{v}^k \leadsto w^{k-1} \leadsto \ldots \leadsto \mathcal{O}.
	\label{eq:4.7}
	\end{equation}
By construction, $z(w^j- \bar{v}^k)= j_+$\,.
However, this does not yet determine $z(w^{k-1}-\mathcal{O})$ to be $k-1$, as claimed in \eqref{eq:4.5}.

\begin{lem}\label{lem:4.2}
For any $1\leq k<n= i(\mathcal{O})$, the first $+$descendant $w^{k-1}$ of $\bar{v}^k$ satisfies claim \eqref{eq:4.5}.
In particular
	\begin{equation}
	w^{k-1}\in\Sigma^{k-1}_+(\mathcal{O}).
	\label{eq:4.8a}
	\end{equation}
\end{lem}

\begin{proof}[\textbf{Proof:}]
By construction of the $+$descendant $w^{k-1}$ of $\bar{v}^k$, we have
	\begin{equation}
	i(w^{k-1})=k-1;
	\label{eq:4.8b}
	\end{equation}
see \eqref{eq:2.7}.
Since $\Sigma^k_+(\mathcal{O})\ni \bar{v}^k \leadsto w^{k-1},$ by \eqref{eq:4.7}, we therefore obtain 
	\begin{equation}
	k=z(\bar{v}^k-\mathcal{O})\geq z(w^{k-1}- \mathcal{O}).
	\label{eq:4.9}
	\end{equation}

On the other hand, $\mathcal{O}\leadsto \bar{v}^k \in \Sigma^k_+ (\mathcal{O}), \ \bar{v}^k \leadsto w^{k-1}$, and Morse-Smale transitivity of $\leadsto$, imply $\mathcal{O} \leadsto w^{k-1}$.
Therefore \eqref{eq:2.15}, for $v:=w^{k-1}$ and \eqref{eq:4.8b} yield
	\begin{equation}
	k-1=i(w^{k-1})\leq z(\mathcal{O}- w^{k-1})= z(w^{k-1}- \mathcal{O}).
	\label{eq:4.10}
	\end{equation}
Together,\eqref{eq:4.9} and \eqref{eq:4.10} leave us with the options $z(w^{k-1}- \mathcal{O})\in \{k-1,k\}$.

Suppose, indirectly, that the bad option $z(w^{k-1}-\mathcal{O})= k$ holds true.
Then $\mathcal{O}<_0 \bar{v}^k< _0w^{k-1}$, by $\bar{v}^k\in \Sigma^k_+(\mathcal{O})$ and $w^{k-1}\in \Sigma^{k-1}_+(\bar{v}^k)$, implies $w^{k-1}\in  \Sigma^k_+(\mathcal{O}). $ 
This contradicts the maximality of $\bar{v}^k$ in $\Sigma^k_+(\mathcal{O})$, at $x=0$, and proves $z(w^{k-1}- \mathcal{O})=k-1$, as claimed in \eqref{eq:4.5}.

Since $\mathcal{O}<_0 \bar{v}^k< _0w^{k-1}$ still holds, we also obtain $z(w^{k-1}-\mathcal{O})= (k-1)_+$\,.
Moreover we recall $\mathcal{O}\leadsto w^{k-1}$.
Together this establishes $w^{k-1}\in\Sigma^{k-1}_+(\mathcal{O})$, as claimed in \eqref{eq:4.8a}, and the lemma is proved.
\end{proof}

\begin{thm} \label{thm:4.3}
In the setting and notation of \eqref{eq:1.29}, \eqref{eq:1.30},
	\begin{equation}
	\bar{v}^k=\underline{v}^k
	\label{eq:4.11}
	\end{equation}
holds for all $0\leq k< n = i(\mathcal{O})$.
\end{thm}

\begin{proof}[\textbf{Proof:}]	
For $k=0$, where $\Sigma^0_+(\mathcal{O})= \{v^0\}$ consists of a single equilibrium anyway, there is nothing to prove.
Therefore consider $1\leq k < n$.
We proceed indirectly and suppose $\bar{v}^k\neq \underline{v}^k$. 
Let $w^j\in \Sigma^j_+(\bar{v}^k)$ denote the $+$descendants of $\bar{v}^k$, as in \eqref{eq:4.7} and in lemma \ref{lem:4.2}.
To reach a contradiction we prove the following three contradictory claims, separately:
	\begin{eqnarray}
	z(w^{k-1}-\underline{v}^k) &=& k-1\,; \label{eq:4.12}\\
	z(w^{k-1}-\underline{v}^k) &<& z(\bar{v}^k- \underline{v}^k)\,; \label{eq:4.13}\\
	z(\bar{v}^k-\underline{v}^k) &\leq& k-1\,. \label{eq:4.14}
	\end{eqnarray}

Among all equilibria in $\mathcal{E}^k_+(\mathcal{O})\subseteq \Sigma^k_+(\mathcal{O})$, we recall  $\underline{v}^k$ is closest to $\mathcal{O}$ at $x=1$, and $\bar{v}^k\in \mathcal{E}^k_+(\mathcal{O})\subseteq \Sigma^k_+(\mathcal{O})$ is most distant at $x=0$; see definitions \eqref{eq:1.29} and \eqref{eq:1.30}.
In particular, $\underline{v}^k \in \mathcal{E}^k_+(\mathcal{O})$ is strictly between $\mathcal{O}$ and $\bar{v}^k$, both, at $x=0$ and $x=1$, by our indirect assumption $\underline{v}^k \neq \bar{v}^k$.
From \eqref{eq:3.5a} and $\underline{v}^k, \bar{v}^k\in \Sigma^k_+(\mathcal{O})$ we also recall $z(\bar{v}^k-\underline{v}^k)\leq k-1$, as claimed in \eqref{eq:4.14}.

Consider the $+$descendants $w^j$ of $\bar{v}^k$, with $j$ of the same even/odd parity as $k$.
Then lemma \ref{lem:2.2}(iii) and (iv) imply
	\begin{eqnarray}
	\bar{v}^k<_1 w^j <_1 \ldots <_1 w^2 <_1 w^0, &\quad& \mathrm{for\ even}~ j<k, \label{eq:4.15}\\
	\bar{v}^k>_1 w^j >_1 \ldots >_1 w^3 >_1w^1, &\quad& \mathrm{for\ odd}~ j<k, \label{eq:4.16}
	\end{eqnarray}
at $x=1$.
Because $\underline{v}^k$ is closest to $\mathcal{O}$ in $\mathcal{E}^k_+(\mathcal{O})\subseteq \Sigma^k_+(\mathcal{O})$, at $x=1$, it lies strictly between $\mathcal{O}$ and $\bar{v}^k\in \mathcal{E}^k_+(\mathcal{O})\subseteq \Sigma^k_+(\mathcal{O}) $ there.
In particular $\underline{v}^k$ fits into \eqref{eq:4.15} and \eqref{eq:4.16} as follows:
	\begin{eqnarray}
	\mathcal{O}<_1\underline{v}^k<_1 \bar{v}^k <_1 w^j <_1 \ldots <_1 w^2 <_1 w^0, &\quad& \mathrm{for\ even}~ j<k, \label{eq:4.17}\\
	\mathcal{O}>_1\underline{v}^k>_1 \bar{v}^k >_1 w^j >_1 \ldots >_1 w^3 >_1 w^1, &\quad& \mathrm{for\ odd}~ j<k. \label{eq:4.18}
	\end{eqnarray}
Moreover $\bar{v}^k\in \Sigma^k_+(\mathcal{O})$, and, by lemma \ref{lem:4.2}, $w^{k-1}\in\Sigma^{k-1}_+(\mathcal{O})$ lie on opposite sides of $\mathcal{O}$, at $x=1$, by opposite parities of $k$ and $k-1$.
Hence $\underline{v}^k$ lies strictly between $\bar{v}^k$ and $w^{k-1}$, at $x=1$.
Therefore zero number dropping of $z(u(t,\cdot) - \underline{v}^k)$ along the heteroclinic orbit $u(t,\cdot),\ \bar{v}^k\leadsto w^{k-1}$, implies $z(\bar{v}^k- \underline{v}^k)> z(w^{k-1}- \underline{v}^k)$, as claimed in \eqref{eq:4.13}.

To prove claim \eqref{eq:4.12}, finally, let $\zeta_j:= z(w^j- \underline{v}^k)\geq0$.
We apply the pigeon hole proposition \ref{eq:2.1}.
First, we note $\zeta_{j-1}< \zeta_j$\,, for all $j=1, \ldots, k-1,$ because $w^j\leadsto w^{j-1}$ and $w^{j-1}, w^j$ are on opposite sides of $\underline{v}^k$, by \eqref{eq:4.17}, \eqref{eq:4.18}.
Since \eqref{eq:4.13} and \eqref{eq:4.14} imply $\zeta_{k-1}\leq k-1$, claim \eqref{eq:4.12} follows from proposition \ref{eq:2.1} which asserts $\zeta_j=j$ for all $0\leq j \leq k-1$.
This proves the theorem by the contradictions \eqref{eq:4.12} -- \eqref{eq:4.14}.
\end{proof}

So far, we have only considered descendants $\underline{v}^k=v^k=\bar{v}^k$ based on the constant sign sequence $\mathbf{s}=++\ldots$ of \eqref{eq:1.26}.  
The four trivial equivalences \eqref{eq:1.27} provide the following extension of minimax theorem \ref{thm:4.3}.

\begin{cor}\label{cor:4.4}
Let $0\leq k< n = i(\mathcal{O}),\ \iota=0,1$, and consider all equilibria in the open hemisphere $\Sigma_+^k(\mathcal{O})$.
Then the equilibrium closest to $\mathcal{O}$, at $x=\iota$, coincides with the equilibrium most distant from $\mathcal{O}$, at the opposite boundary $x=1-\iota$.
The same statement holds true among the equilibria in the opposite open hemisphere $\Sigma_-^k(\mathcal{O})$.

In other words, the $h_\iota$-closest equilibrium to $\mathcal{O}$ is $h_{1-\iota}$-most distant, in the same hemisphere $\Sigma_\pm^k(\mathcal{O})$.

\end{cor}


\section{Discussion}\label{sec5}

In this final section we explore and illustrate what our main theorem \ref{thm:1.2} does, and does not, say.
We first review the most celebrated Sturm global attractor, the $n$-dimensional Chafee-Infante attractor $\mathcal{CI}_n$ \cite{chin74}.
Contrary to the common approach, which starts from a symmetric cubic nonlinearity, an ODE discussion of equilibria, and the time map of their pendulum boundary value problem, we start from an abstract description of $\mathcal{CI}_n$ as the $n$-dimensional Sturm attractor with the minimal number $N=2n+1$ of equilibria. 
We derive the associated signed Thom-Smale complex next, in this much more general context.
Invoking lemma \ref{lem:2.2} will then provide the well-known shooting meander, and the Sturm permutation, of the Chafee-Infante attractor $\mathcal{CI}_n$\,.

In the second part of our discussion, we present examples of abstract signed regular complexes which are 3-balls.
We first adapt the general recipe of theorem \ref{thm:1.2} for the construction of the associated boundary orders $h_0, h_1$ to the special case of Sturmian 3-balls, in the spirit of \cite{firo3d-1, firo3d-2}. See theorem \ref{thm:5.2} and definitions \ref{def:5.1}, \ref{def:5.3}, \ref{def:5.4}.
Going beyond the introductory example of fig.~\ref{fig:1.0}, we address three further examples for illustration.

Our first new example, in fig.~\ref{fig:5.2}, constructs $h_0, h_1$ and the permutation $\sigma = h_0^{-1}\circ h_1$ for the unique Sturm solid tetrahedron with two faces in each hemisphere.
The locations of the poles $\mathbf{N}, \mathbf{S}$ along the 4-edge meridian circle turn out to be edge-adjacent, necessarily.

The second example, fig.~\ref{fig:5.3}, deviates from the above tetrahedral signed Thom-Smale complex, by misplacing the South pole and reversing the orientation of a single edge, along the same meridian circle as before. 
The new location of the South pole is not edge-adjacent to the North pole, along the meridian, but diametrically opposite, instead.
Still, our recipe succeeds to construct a Sturm permutation $\sigma:=h_0^{-1}\circ h_1$\,.
The associated Sturm global attractor, however, fails to be the regular solid tetrahedron.

The third example, fig.~\ref{fig:5.4}, starts from a signed regular solid octahedron complex with antipodal pole locations. 
It was first observed in \cite{firo14} that such antipodal octahedra cannot be of Sturm type.
Nevertheless, our construction of the total orders $h_0$ and $h_1$ still succeeds, at first sight.
The permutation $\sigma:=h_0^{-1}\circ h_1$\,, however, fails to define a meander.
We conclude with comments on the still elusive goal of a geometric characterization of \emph{all} Sturmian Thom-Smale complexes $\mathcal{C}^s_f$ as abstract signed regular cell complexes $\mathcal{C}^s$.

The $n$-dimensional \emph{Chafee-Infante attractor} $\mathcal{CI}_n$ is usually presented as the Sturm global attractor of 
	\begin{equation}
	u_t = u_{xx} + \lambda^2 u(1-u^2)
	\label{eq:CI.1}
	\end{equation}
on the unit interval $0<x<1$, with parameter $0<(n-1)\pi < \lambda < n\pi$, cubic nonlinearity, and for Neumann boundary conditions. 
See \cite{chin74} for the closely related original Dirichlet setting. 
Geometrically, $\mathcal{CI}_0$ can be thought of as the single trivial equilibrium $\mathcal{O}\equiv0$, and $\mathcal{CI}_n$ is the one-dimensionally unstable double cone suspension of $\mathcal{CI}_{n-1}$\,, inductively for $n>0$. 
See \cite{he85, fi94}. 
The double cone suspension is a generalization of the passage to a sphere $\Sigma^n$ from its equator $\Sigma^{n-1}$\,, of course.

The Chafee-Infante attractor $\mathcal{CI}_n$ can also be characterized as the Sturm attractor of maximal dimension $n=(N-1)/2$, for any (necessarily odd) number $N$ of equilibria.
Equivalently, $\mathcal{CI}_n$ is the $n$-dimensional Sturm attractor with the minimal number $N=2n+1$ of equilibria. 
See \cite{fi94}.

Let us start from that latter characterization, abstractly, without any ODE or PDE recurrence to the explicit equation \eqref{eq:CI.1} whatsoever.
We will first derive the relevant part of the signed Thom-Smale complex $\mathcal{C}^s$ for the $n$-dimensional Chafee-Infante attractor $\mathcal{CI}_n$ with $N=2n+1$ equilibria. 
We will then determine the two boundary orders $h_\iota$, and arrive at the meander permutation $\sigma=h_0^{-1}\circ h_1$ of the Chafee-Infante attractor  $\mathcal{CI}_n$, without any ODE  shooting analysis of the equilibrium boundary value problem \eqref{eq:1.2}.

Fix any dimension $n\geq 1$.
By assumption, $\mathcal{C}^s$ contains at least one cell $c_\mathcal{O}$ of dimension $n$.
The relevant part of the signed Thom-Smale complex $\mathcal{C}^s$ is the signed hemisphere decomposition $\Sigma_\pm^k=\Sigma_\pm^k(\mathcal{O}),\ 0\leq k<n,$ of $\Sigma^{n-1}=\partial W^u(\mathcal{O})$; see \eqref{eq:1.8}--\eqref{eq:1.11}.
Each of the $2n$ open hemispheres $\Sigma_\pm^k$ must contain at least one equilibrium $v_\pm^k\neq\mathcal{O}$.
For the minimal number $N-1=2n$ of remaining equilibria, this proves 
	\begin{equation}
	\label{eq:CI.2}
	\Sigma_\pm^k(\mathcal{O}) = W^u(v_\pm^k)\,.
	\end{equation}
In particular all minimal or maximal equilibria $\underline{v}^k,\ \bar{v}^k$ defined in \eqref{eq:1.29}, \eqref{eq:1.30} coincide, trivially, in each single-element equilibrium set $\mathcal{E}_\pm^k = \{v_\pm^k\}$ of the hemisphere decomposition \eqref{eq:CI.2}.

\begin{figure}[t!]
\centering \includegraphics[width=1.0\textwidth]{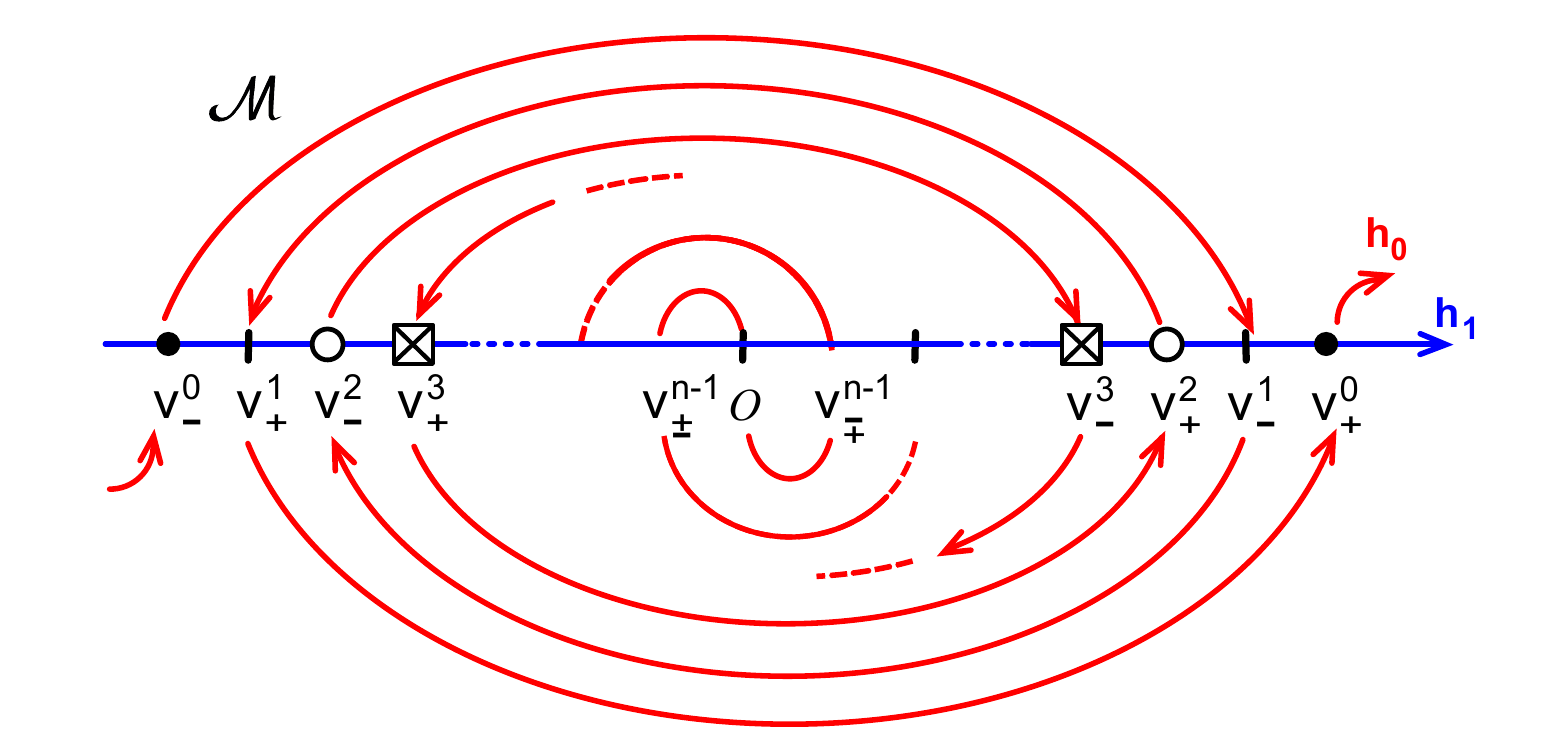}
\caption{\emph{
The meander $h_0$ (red) of the $n$-dimensional Chafee-Infante attractor \eqref{eq:CI.1}.
Note the two collections of nested arcs, one above and one below the horizontal $h_1$ axis (blue). 
The outermost arcs begin and terminate at the poles $v_\pm^0$\,, respectively, and the innermost arcs involve $\mathcal{O}$. 
Because the two nests are shifted with respect to each other, their arcs join to form a meander curve, in fact a double spiral, with an inflection at the central equilibrium $\mathcal{O}$.
The neighbors $v_\pm^{n-1}$ of $\mathcal{O}$ are $v_+^{n-1}$ to the left and $v_-^{n-1}$ to the right, if $n-1$ is odd. 
For even $n-1$ the subscript signs are swapped.
}}
\label{fig:CI.1}
\end{figure}

For the $\mathbf{s}$ descendants $v^k(\mathbf{s}) \in \Sigma_{s_k}^k$\,, defined in \eqref{eq:1.18}, \eqref{eq:1.19a}, the hemisphere characterization \eqref{eq:CI.2} therefore implies
	\begin{equation}
	\label{eq:CI.4}
	v^k(\mathbf{s}) =v_{s_k}^k\,.
	\end{equation}
In particular, theorems \ref{thm:3.1} and \ref{thm:4.3} become trivial.

Invoking lemma \ref{lem:2.2}(ii) for constant $\mathbf{s}=++\ldots$ and $\mathbf{s}=--\ldots$\,, however, implies the total order
	\begin{equation}
	\label{eq:CI.5}
	v_-^0 <_0  v_-^1 <_0 v_-^2 <_0 \ldots <_0 v_-^{n-1} <_0 \mathcal{O} <_0 v_+^{n-1} <_0 \ldots <_0 v_+^2 <_0 v_+^1 <_0 v_+^0 \,.
	\end{equation}
of all equilibria, at the boundary $x=0$.
Here we have applied the trivial equivalence $u\mapsto-u$ of \eqref{eq:1.27}, as in the end of section \ref{sec2}, to also derive the ordering of $v_-^k$ at $x=0$.
This determines $h_0$\,.

To determine the total order $h_1$ of all equilibria at the opposite boundary $x=1$, we invoke the trivial equivalence $x\mapsto 1-x$ of \eqref{eq:1.27}.
Since this equivalence swaps $s_j$ with $(-1)^j s_j$\,, and $h_0$ with $h_1$\,, we can apply \eqref{eq:CI.4} and \eqref{eq:CI.5} to conclude
	\begin{equation}
	\label{eq:CI.6}
	v_-^0 <_1 v_+^1 <_1 v_-^2 <_1 v_+^3 <_1 \ldots <_1 \mathcal{O} <_1  \ldots <_1 v_-^3 <_1 v_+^2 <_1 v_-^1 <_1 v_+^0 \,.
	\end{equation}
Indeed the equilibria $v_+^k$ appear below $\mathcal{O}$ for odd $k$, in increasing order, and above $\mathcal{O}$ for even $k$, in decreasing order. 
For the intercalating equilibria $v_-^k$ the parities of $k$ are swapped. 

See fig.~\ref{fig:CI.1} for the resulting meander $h_0$\,, based on the orders  \eqref{eq:CI.5}, \eqref{eq:CI.6}.
Notably the meander consists of two collections of nested arcs, one above and one below the horizontal $h_1$ axis. 
The two outermost arcs join the two poles $v_\pm^0$ with $v_\pm^1$\,, respectively.
The two innermost arcs join the inflection point $\mathcal{O}$ with $v_\pm^{n-1}$.
In cycle notation, the Chafee-Infante Sturm permutation $\sigma=h_0^{-1}\circ h_1$ of the $N=2n+1$ equilibria follows easily from
	\eqref{eq:CI.5}, \eqref{eq:CI.6} to be the involution
	\begin{equation}
	\label{eq:CI.7}
	\sigma = (2\ \ 2n)\ (4\ \ 2n-2) \ldots (2\lfloor\tfrac{n}{2}\rfloor\ \ 2\lfloor\tfrac{n+3}{2}\rfloor)\,.
	\end{equation}
Here $\lfloor\cdot\rfloor$ denotes the integer floor function.

Consider the four constant and alternating sign sequences $\mathbf{s}$.
The boundary orders \eqref{eq:CI.5} and \eqref{eq:CI.6} then make it easy to determine the four corresponding heteroclinic staircases $v^j(\mathbf{s})=v^j_{s_j}\,,\ 0\leq j < n$, in \eqref{eq:1.19b} descending from $\mathcal{O}$; see also \eqref{eq:CI.4}.
The $+$descendants, in fact, define the complete part of the meander path $h_0$ after $\mathcal{O}$, from $\mathcal{O}$ to $v_+^0$\,.
The $-$descendants define the part of the meander $h_0$ before $\mathcal{O}$ in reverse order, back from $\mathcal{O}$ to $v_-^0$\,.
Similarly the alternating descendants $\mathbf{s}=\ldots-+$ define the horizontal part, along the $h_1$-axis, from $\mathcal{O}$ to $v_+^0$\,.
The final choice $\mathbf{s}=\ldots+-$ of alternating descendants defines the horizontal part, backwards along the $h_1$-axis, from $\mathcal{O}$ to $v_-^0$\,.

A posteriori, of course, the Sturm permutation \eqref{eq:CI.7} of the $n$-dimensional Chafee-Infante attractor $\mathcal{CI}_n$ also determines the zero numbers and Morse indices
	\begin{equation}
	\label{eq:CI.3}
	z(v_\pm^j-v_\pm^k) = z(v_\pm^j-v_\mp^k)= z(v_\pm^j-\mathcal{O})=j_\pm \,, \qquad i(v_\pm^j)=j\,,
	\end{equation}
for all $0\leq j<k<n$.
See \eqref{eq:1.7a} and \eqref{eq:1.7b}.
Conversely, a priori knowledge of all signed zero numbers determines the Sturm permutation $\sigma=h_0^{-1}\circ h_1$\,, in any Thom-Smale complex. 
Indeed, the signs of  $z$ immediately determine the total order $h_0$ of all equilibria $v_k$\,, at $x=0$. 
Keeping the even/odd parity of $k$ in mind, the same signs determine the total order $h_1$ of all equilibria $v_k$\,, at $x=1$.

For general abstract signed regular complexes, however, matters are not that simple.
The prescribed hemisphere signs do not keep track of the relative boundary orders of all barycenter pairs $v_j\,,\, v_k$\,, explicitly.
This information is only available for those pairs $v_j\,,\, v_k$ for which one barycenter is in the cell boundary of the other.
(A posteriori, in other words, these are the heteroclinic pairs $v_k \leadsto v_j$ in the resulting Sturm attractor.)
How to extend this partial order to the two total orders $h_0\,, h_1$\,, uniquely for each of them, was the main result of the present paper.
See theorem \ref{thm:1.2}.

To illustrate this point further, we turn to 3-ball Sturm attractors $\mathcal{A}_f$ next. A purely geometric characterization of their signed hemisphere decompositions \eqref{eq:1.8}--\eqref{eq:1.11} has been achieved in \cite{firo3d-1, firo3d-2}; see also \cite{firo3d-3} for many examples. 
Dropping all Sturmian PDE interpretations, we defined 3-cell templates, abstractly, in the class of signed regular cell complexes $\mathcal{C}^s$ and without any reference to PDE terminology or to dynamics.
Recall fig.~\ref{fig:1.0}(b) for a first illustration.

\begin{defi}\label{def:5.1}
A finite signed regular cell complex $\mathcal{C}^s = \bigcup_{v \in \mathcal{E}} c_v$ is called a \emph{3-cell template} if the following four conditions all hold for the hemispheres $\Sigma_\pm^j=\Sigma_\pm^j(\mathcal{O})$ and descendants $v^j=v^j(\mathbf{s})$ of $\mathcal{O}$, according to definition \ref{def:1.1}.
	\begin{itemize}
	\item[(i)] $\mathcal{C}^s = \mathrm{clos}\, c_{\mathcal{O}}= S^2 \,\dot{\cup}\, c_{\mathcal{O}}$ is the closure of a single 3-cell $c_{\mathcal{O}}$\,.
	\item[(ii)] The 1-skeleton $(\mathcal{C}^s)^1$ of $\mathcal{C}^s$ possesses a \emph{bipolar orientation} from a pole vertex $\mathbf{N}:=\Sigma_-^0$ (North) to a pole vertex $\mathbf{S}:=\Sigma_+^0$ (South), with two disjoint directed \emph{meridian paths} $\mathbf{EW}:=\Sigma_-^1$ (EastWestern) and $\mathbf{WE}:=\Sigma_+^1$ (WestEastern) from $\mathbf{N}$ to $\mathbf{S}$.
The meridians decompose the boundary sphere $S^2$ into remaining hemisphere components $\Sigma_-^2:=\mathbf{W}$ (West) and $\Sigma_+^2:=\mathbf{E}$ (East).
	\item[(iii)] Edges are directed towards the meridians, in $\Sigma_-^2$\,, and away from the meridians, in $\Sigma_+^2$\,, at edge boundary vertices on the meridians other than the poles $\Sigma_\pm^0$\,.
	\item[(iv)] For $\iota=0,1$, let $w_\pm^\iota$ denote the first descendants $v^2(\mathbf{s})$ of $\mathcal{O}$, as defined in \eqref{eq:1.22}--\eqref{eq:1.25}.
Then the boundaries of the 2-cells of $w_-^\iota$ and of $w_+^{1-\iota}$ overlap in at least one edge, along the appropriate meridian $\Sigma_\pm^1$\,, between the respective second descendants $v^1(\mathbf{s})$ of $\mathcal{O}$.
	\end{itemize}
\end{defi}

We recall here that an edge orientation of the 1-skeleton $(\mathcal{C}^s)^1$ is called bipolar if it is without directed cycles, and with a single ``source'' vertex $\mathbf{N}$ and a single ``sink'' vertex $\mathbf{S}$ on the boundary of $\mathcal{C}^s$.
Here ``source'' and ``sink'' are understood, not dynamically but, with respect to edge direction.
The edge orientation of any 1-cell $c_v$ runs from $\Sigma_-^0(v)$ to $\Sigma_+^0(v)$.
The most elementary hemi-``sphere'' decomposition of 1-cells, in other words, can simply be viewed as an edge orientation.
Bipolarity is a local and global compatibility condition for these orientations which, in particular, forbids directed cycles.

By definition \ref{def:1.1} of descendants, the 2-cells $\mathbf{NE}$ of $w_-^0$ and $\mathbf{SW}$ of $w_+^1$  denote the unique faces in $\mathbf{W}$, $\mathbf{E}$ which contain the first, last edge of the meridian $\mathbf{WE}$ in their boundary, respectively.
In definition \ref{def:5.1}(iv), the boundaries of $\mathbf{NE}$ and $\mathbf{SW}$ are required to overlap in at least one shared edge along that meridian $\mathbf{WE}$.

Similarly, the 2-cells $\mathbf{NW}$ of $w_-^1$ and $\mathbf{SE}$ of $w_+^0$  denote the unique faces in $\mathbf{W}$, $\mathbf{E}$, respectively, which contain the first, last edge of the other meridian $\mathbf{EW}$ in their boundary, respectively.
The boundaries of $\mathbf{NW}$ and $\mathbf{SE}$ are required to overlap in at least one shared edge along that meridian $\mathbf{EW}$.

The main result of \cite{firo3d-1, firo3d-2}, in our language of descendants, reads as follows. 

\begin{thm}\label{thm:5.2}
\emph{\cite[theorems~1.2 and~2.6]{firo3d-2}}.
A finite signed regular cell complex $\mathcal{C}^s$ coincides with the signed Thom-Smale dynamic complex $c_v= W^u(v) \in \mathcal{C}^s_f$ of a 3-ball Sturm attractor $\mathcal{A}_f$ if, and only if, $\mathcal{C}^s$ is a 3-cell template.
\end{thm}

In \cite[theorem~4.1]{firo3d-1} we proved that the signed Thom-Smale complex $\mathcal{C}^s$:= $\mathcal{C}^s_f$ of any Sturm 3-ball $\mathcal{A}_f$ indeed satisfies properties (i)--(iv) of definition~\ref{def:1.1}.
In our example of fig.~\ref{fig:1.0} this simply means the passage (a) $\Rightarrow$ (b).
In general, the 3-cell property~(i) of $c_{\mathcal{O}} = W^u(\mathcal{O})$ is obviously satisfied.
The bipolar orientation~(ii) of the edges $c_v$ of the 1-skeleton $(\mathcal{C}^s)^1$ is a necessary condition, for Sturm signed Thom-Smale complexes $\mathcal{C}^s=\mathcal{C}^s_f$\,. 
Indeed, acyclicity of the orientation of edges $c_v$\,, alias the one-dimensional unstable manifolds $c_v = W^u(v)$ of $i(v)=1$ saddles $v$, simply results from the strictly monotone $z=0$ ordering of each edge $c_v$\,: 
from the lowest equilibrium vertex $\Sigma_-^0 (v)$ in the closure $\bar{c}_v$ to the highest equilibrium vertex $\Sigma_+^0(v)$. 
The ordering is uniform for $0 \leq x \leq 1$, and holds at $x \in \{0,1\}$, in particular.
The poles $\mathbf{N}$ and $\mathbf{S}$ indicate the lowest and highest equilibrium, respectively, in that order.
Again we refer to fig.~\ref{fig:1.0} for an illustrative example.

In our present language, properties (i) and (ii) thus describe the 3-cell template as a hemisphere decomposition of the boundary $\Sigma^2 = \partial c_\mathcal{O}$ of the single 3-cell $\mathcal{O}$.
The meridian cycle is the boundary $\Sigma^1$ of the two-dimensional fast unstable manifold $W^2(\mathcal{O})$.
In addition, property (ii) compatibly concatenates the 1-cell orientations, equivalent to the strong monotonicity of the defining heteroclinic orbits in each cell, to a global bipolar orientation of the 1-skeleton $(\mathcal{C}^s)^1$.

Properties (iii) and (iv) are far less obvious constraints for an abstract signed regular 3-cell complex $\mathcal{C}^s$ to qualify as a signed 3-cell Thom-Smale complex $\mathcal{C}^s = \mathcal{C}^s_f$ of Sturm type.

The main result of our present paper, theorem \ref{thm:1.2}, determines the boundary paths $h_0, h_1$ from a signed 3-cell template $\mathcal{C}^s$ which is already known to be the signed Thom-Smale complex $\mathcal{C}^s=\mathcal{C}_f^s$ of a 3-ball Sturm attractor $\mathcal{A}_f$\,. 
In our example, this describes the passage from fig.~\ref{fig:1.0}(b) to fig.~\ref{fig:1.0}(d). 
We describe an equivalent practical simplification of this construction next, in terms of an SZS-pair $(h_0, h_1)$ of Hamiltonian paths $h_\iota$: $\lbrace 1, \ldots , N\rbrace \rightarrow \mathcal{E}$; see  \cite[section~2]{firo3d-3} for further details.

To prepare our construction, we first consider planar regular cell complexes $\mathcal{C}$, abstractly, with a bipolar orientation of the 1-skeleton $\mathcal{C}^1$.
Here bipolarity requires that the unique poles $\mathbf{N}$ and $\mathbf{S}$ of the orientation are located at the boundary of the embedded regular complex $\mathcal{C} \subseteq \mathbb{R}^2$.

To traverse the vertices $v \in \mathcal{E}$ of a planar complex $\mathcal{C}$, in two different ways, we construct a pair of directed Hamiltonian paths
	\begin{equation}
	h_0, h_1: \quad \lbrace 1, \ldots , N \rbrace \rightarrow \mathcal{E}
	\label{eq:5.1}
	\end{equation}
as follows.
Let $\mathcal{O}$ indicate any source, i.e. the barycenter of any 2-cell  face $c_{\mathcal{O}}$ in $\mathcal{C}$.
(We temporarily deviate from the standard 3-ball notation, here, to emphasize analogies with the passage of $h_\iota$ through a 3-cell.)
By planarity of $\mathcal{C}$ the bipolar orientation of $\mathcal{C}^1$ defines unique extrema on the boundary circle $\partial c_{\mathcal{O}}$ of each 2-cell $c_\mathcal{O}$.
Let $w_-^0$ denote the barycenter on $\partial c_{\mathcal{O}}$ of the first edge to the right of the minimum, and $w_+^0$ the first barycenter to the left of the maximum.
See fig.~\ref{fig:5.1}.
Similarly, let $w_-^1$ be the first barycenter to the left of the minimum, and $w_+^1$ first to the right of the maximum.
Then the following definition serves as our practical construction recipe for the pair $(h_0,h_1$).

\begin{figure}[t!]
\centering \includegraphics[width=0.75\textwidth]{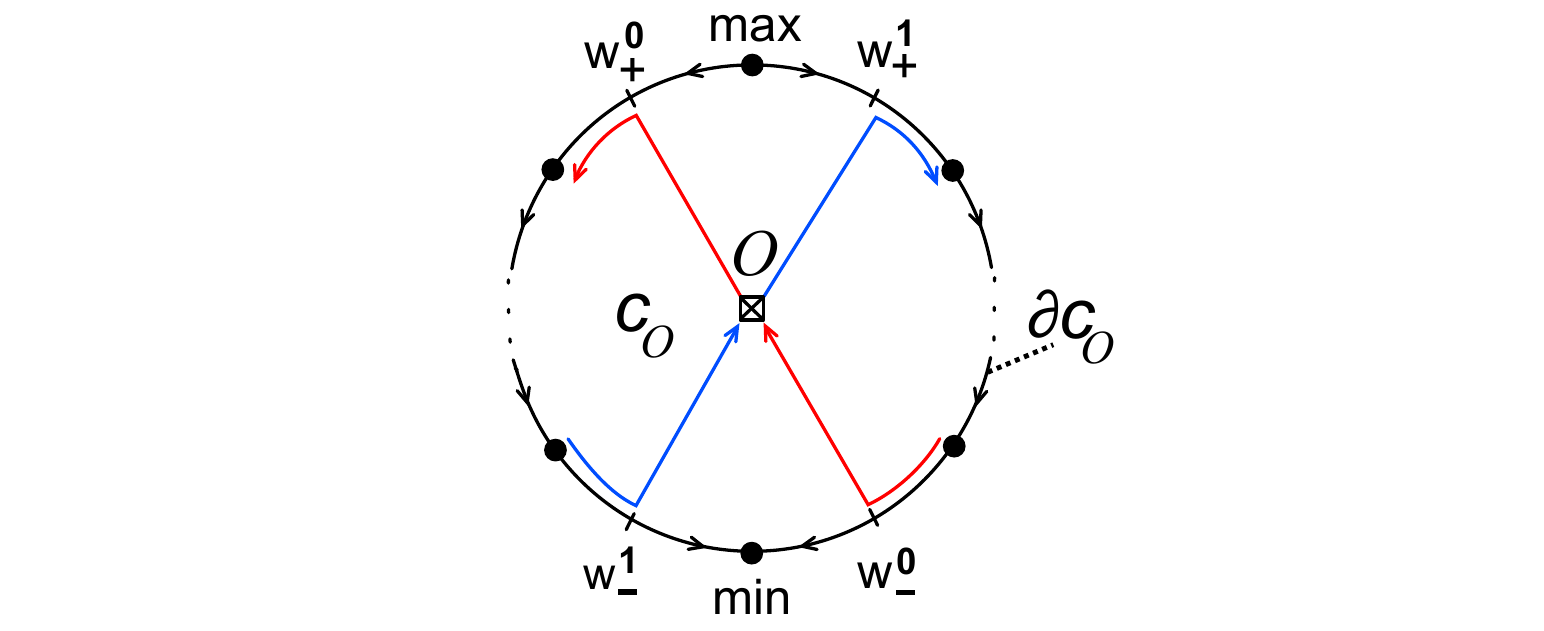}
\caption{\emph{
Traversing a face vertex $\mathcal{O}$ by a ZS-pair $h_0, h_1$\,.
Note the resulting shapes ``Z'' of $h_0$ (red) and ``S'' of $h_1$ (blue).
The paths $h_\iota$ may also continue into adjacent neighboring faces, beyond $w_\pm^\iota$\,, without turning into the face boundary $\partial c_{\mathcal{O}}$.
}}
\label{fig:5.1}
\end{figure}

\begin{defi}\label{def:5.3}
The bijections $h_0, h_1$ in \eqref{eq:5.1} are called a \emph{ZS-pair} $(h_0, h_1)$ in the finite, regular, planar and bipolar cell complex $\mathcal{C} = \bigcup_{v \in \mathcal{E}} c_v$ if the following three conditions all hold true:
	\begin{itemize}
	\item[(i)] $h_0$ traverses any face $c_\mathcal{O}$ from $w_-^0$ to $w_+^0$;
	\item[(ii)] $h_1$ traverses any face $c_\mathcal{O}$ from $w_-^1$ to $w_+^1$
	\item[(iii)] both $h_\iota$ follow the same bipolar orientation of the 1-skeleton $\mathcal{C}^1$, unless defined by (i), (ii) already.
	\end{itemize}
We call $(h_0,h_1)$ an \emph{SZ-pair}, if $(h_1, h_0)$ is a ZS-pair, i.e. if the roles of $h_0$ and $h_1$ in the rules (i) and (ii) of the face traversals are swapped.
\end{defi}

Properties (i)-(iii) of definition \ref{def:5.3} of a ZS-pair $(h_0,h_1)$ are equivalent to our present theorem \ref{thm:1.2}, in the language of descendants.
Indeed, we just have to define the signed hemisphere decomposition of each planar face $c_\mathcal{O}$ such that $\Sigma_-^1(\mathcal{O})$ appears to the right of the boundary minimum $\Sigma_-^0(\mathcal{O})$ or, equivalently, of the boundary maximum $\Sigma_+^0(\mathcal{O})$. 
Similarly, $\Sigma_+^1(\mathcal{O})$ appears to the left; see fig.~\ref{fig:5.1}. 

Of course, this choice, together with the bipolar orientation and the boundary extrema on each cell, identifies the abstract planar regular cell complex $\mathcal{C}$ as a signed regular cell complex $\mathcal{C}^s$, with certain global rules on the cell signatures.
Indeed, bipolarity serves as an additional global constraint, necessarily satisfied by planar Thom-Smale complexes of Sturm type.
In addition, note how shared edges between adjacent faces receive opposite signatures from either face.
For an SZ-pair, in contrast, we just have to swap the signature roles of $\Sigma_\pm^1(\mathcal{O})$, in every 2-cell. 
The planar trilogy \cite{firo08, firo09, firo10} contains ample material and examples on the planar case.

After these preparations we can now return to the general 3-cell templates $\mathcal{C}^s$ of definition \ref{def:5.1} and define the SZS-pair $(h_0,h_1)$ associated to $\mathcal{C}^s$.

\begin{defi}\label{def:5.4}
Let $\mathcal{C}^s = \bigcup_{v \in \mathcal{E}} c_v$ be a 3-cell template with oriented 1-skeleton $(\mathcal{C}^s)^1$, poles $\mathbf{N}, \mathbf{S}$, hemispheres $\mathbf{W}, \mathbf{E}$, and meridians $\mathbf{EW}$, $\mathbf{WE}$.
A pair $(h_0, h_1)$ of bijections $h_\iota$: $ \lbrace 1, \ldots , N \rbrace \rightarrow \mathcal{E}$ is called \emph{the SZS-pair assigned to} $\mathcal{C}^s$ if the following conditions hold.
\begin{itemize}
	\item[(i)] The restrictions of range $h_\iota$ to $\mathrm{clos}\,\mathbf{W}$ form an SZ-pair $(h_0, h_1)$, in the closed Western hemisphere.
The analogous restrictions form a ZS-pair $(h_0,h_1)$ in the closed Eastern hemisphere $\mathrm{clos}\,\mathbf{E}$.
See definitions \ref{def:5.1} and \ref{def:5.3}.
	\item[(ii)] In the notation of definition \ref{def:5.1}(iv) for the descendants $w_\pm^\iota$ of $\mathcal{O}$, and for each $\iota \in \{0,1\}$, the permutation $h_\iota$ traverses $w_-^\iota, \mathcal{O}, w_+^\iota$\,, successively.
	\end{itemize}
The swapped pair $(h_1,h_0)$ is called \emph{the ZSZ-pair of} $\mathcal{C}^s$.
\end{defi}

See fig.~\ref{fig:1.0} for a specific example.
Condition (i) identifies the closures of the open hemispheres $\mathbf{W}=\Sigma_-^2(\mathcal{O})$ and $\mathbf{E}=\Sigma_+^2(\mathcal{O})$ as the signed Thom-Smale complexes of planar Sturm attractors.
Note how opposite hemispheres receive opposite planar orientation, in fig.~\ref{fig:1.0}(b).
As a consequence, any shared meridian edge $c_v$ in $\Sigma_\pm^1(\mathcal{O})$ receives the same sign from the planar orientation of its two adjacent faces in either signed hemisphere.

Given the Sturm signed Thom-Smale complex of fig.~\ref{fig:1.0}(b), with the orientation of the 1-skeleton determined by the poles $\mathbf{N}=1$ and $\mathbf{S}=9$, we thus arrive at the SZS-pair $(h_0,h_1)$ indicated there. 
This determines the boundary orders of the equilibria in fig.~\ref{fig:1.0}(d).
The meander in fig.~\ref{fig:1.0}(c) is then based on the Sturm permutation $\sigma=h_0^{-1}\circ h_1$\,, as usual.

In summary, theorem \ref{thm:1.2} and, for 3-cell templates equivalently, definition \ref{def:5.4} reconstruct the same generating Hamiltonian paths $h_0, h_1$\,, and hence the same generating Sturm permutation, of any 3-cell template.

\begin{figure}[p!]
\centering \includegraphics[width=1.0\textwidth]{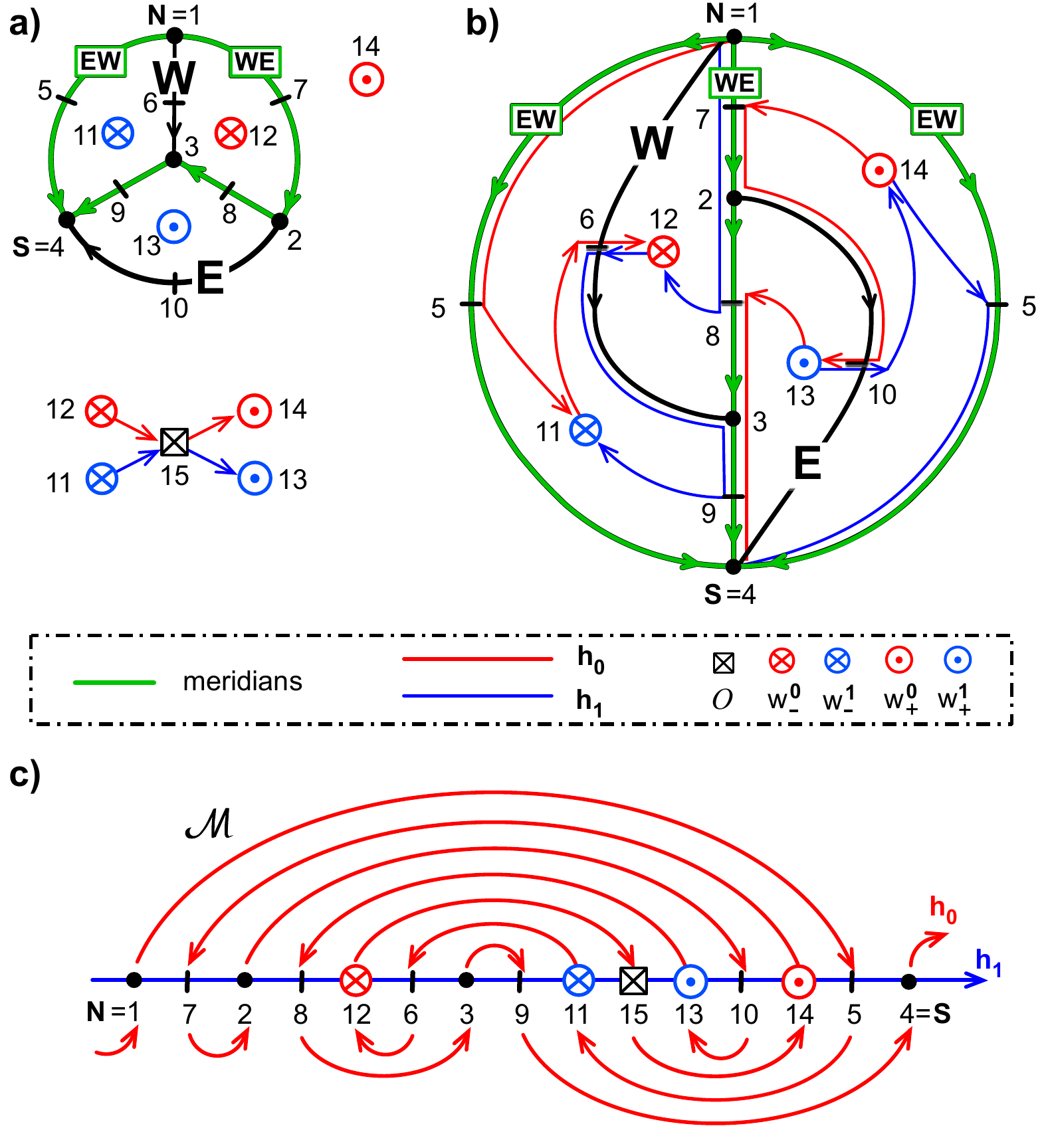}
\caption{\emph{
The Sturm tetrahedron 3-ball with 2+2 faces in the hemispheres $\mathbf{W} = \Sigma_-^2(\mathcal{O})$ and $\mathbf{E} = \Sigma_+^2(\mathcal{O})$. 
(a) Equilibrium labels $\mathcal{E}=\{1,\ldots, 15\}$, bipolar orientation of the 1-skeleton with poles $\mathbf{N}=1$, $\mathbf{S}=4$ which are edge-adjacent along the meridian circle (green), and hemisphere decomposition. 
The face with barycenter 14 is drawn as the exterior, in the 1-point compactification of the plane.
The meridians are indicated as $\mathbf{EW} = \Sigma_-^1(\mathcal{O})$ and $\mathbf{WE} = \Sigma_+^1(\mathcal{O})$. 
See the legend for the predecessors and successors $w_\pm^\iota$ of $\mathcal{O}=15$.
(b) The SZS-pair of Hamiltonian paths $h_0$ (red) and $h_1$ (blue). 
Here we identify the right and left copies of the meridian $\mathbf{EW}$.
See definition \ref{def:5.3}, for the $h_\iota$ predecessors and successors $w_\pm^\iota$ of $\mathcal{O}$, and definition \ref{def:5.4}, for the remaining paths in the respective hemispheres. 
See also \eqref{eq:5.2} for the resulting paths $h_0, h_1$\,.
(c) The dissipative Morse meander defined by the label-independent Sturm permutation $\sigma=h_0^{-1}\circ h_1$; see also \eqref{eq:5.3}.
}}
\label{fig:5.2}
\end{figure}

In the general case, not restricted to 3-balls, we have assumed that the signed regular complex $\mathcal{C}^s=\mathcal{C}^s_f$ is presented as a signed Thom-Smale complex of Sturm type, from the start.
In particular, all hemisphere signs were determined by the signed zero numbers.
We have then described the precise relation between that signed complex $\mathcal{C}^s=\mathcal{C}^s_f$ and the boundary orders, at $x=\iota=0,1$, of the paths $h_\iota$ traversing the complex.
In particular we have proved that the signed Thom-Smale complex $\mathcal{C}^s=\mathcal{C}^s_f$ determines the Sturm permutation $\sigma=\sigma_f$\,, uniquely.
Conversely, abstract Sturm permutations determine their signed Thom-Smale complex, uniquely.
See \cite{firo96, firo00, firo13, firo3d-1}.
This provides a 1-1 correspondence between Sturm permutations and their signed Thom-Smale complexes.

In general, however, we are still lacking a purely geometric characterization of those signed regular cell complexes $\mathcal{C}^s$ which arise as signed Thom-Smale complexes $\mathcal{C}^s=\mathcal{C}^s_f$ of Sturm type. 
Indeed, the characterization by theorem \ref{thm:5.2} covers 3-cell templates $\mathcal{O}$, only.

Three difficulties may arise in an attempt to realize a given signed regular cell complex $\mathcal{C}^s$ as a Sturm complex $\mathcal{C}^s=\mathcal{C}^s_f$\,. 
First, the recipe of theorem \ref{thm:1.2} might fail to provide Hamiltonian paths $h_0,h_1$\,. 
For example, the same barycenter $w$ of an $(n-1)$-cell may be identified as the successor $w=w_+^\iota$ of the barycenters $\mathcal{O}$ and $\mathcal{O}'$ of two different $n$-cells, for the same directed path $h_\iota$\,. 
Or that ``path'' might turn out to contain cyclic connected components, instead of defining a single Hamiltonian path which visits all barycenters. 
Second, even if both paths turn out to be Hamiltonian, from ``source'' $\mathbf{N}$ to ``sink'' $\mathbf{S}$, the resulting permutation $\sigma = h_0^{-1}\circ h_1$ may fail to define a Morse meander -- precluding any realization in the Sturm PDE setting \eqref{eq:1.1}. 
Third, and even if we prevail against both obstacles, the lucky original signed regular complex $\mathcal{C}^s$ may fail to coincide, isomorphically, with the signed Thom-Smale complex $\mathcal{C}^s_f$ associated to the thus constructed Sturm permutation $\sigma=\sigma_f$\,.

\begin{figure}[p!]
\centering \includegraphics[width=1.0\textwidth]{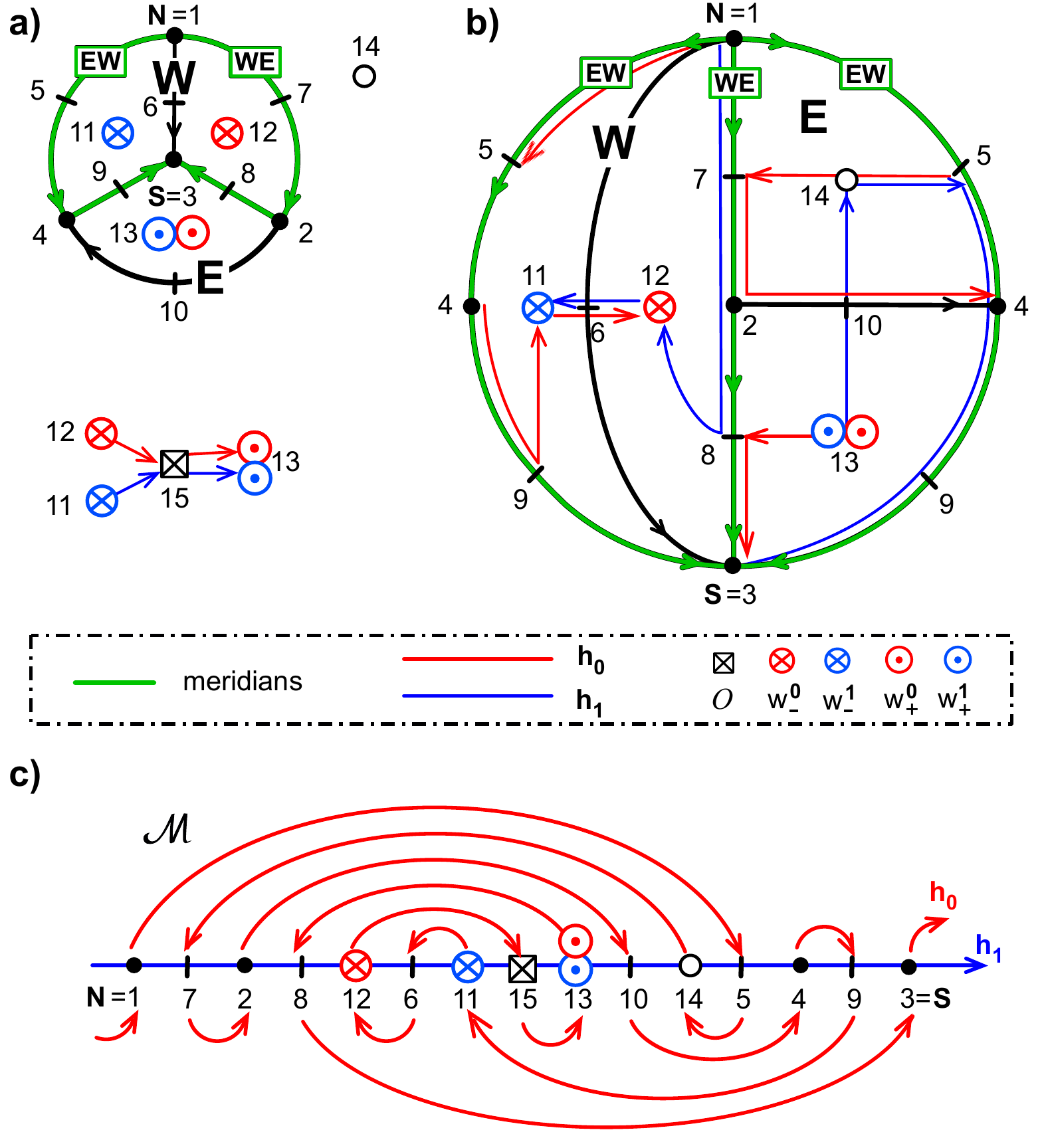}
\caption{\emph{
A signed regular tetrahedron 3-ball complex which is not Sturm. We use the same decomposition into hemispheres $\mathbf{W} = \Sigma_-^2(\mathcal{O})$ and $\mathbf{E} = \Sigma_+^2(\mathcal{O})$ with 2+2 faces as in fig.~\ref{fig:5.2}. Only the edge-adjacent poles have been replaced by $\mathbf{N}=1$, $\mathbf{S}=3$, which are not edge-adjacent along the meridian circle (green). 
(a) Adapted bipolar orientation of the 1-skeleton, and hemisphere decomposition. 
Only the orientation condition (iii) of definition \ref{def:5.1} is violated, necessarily, by the orientation of edge 10  in the hemisphere $\mathbf{E}$.
(b) The SZS-pair $h_0$ (red) and $h_1$ (blue), constructed according to definitions \ref{def:5.3} and \ref{def:5.4}, still provides Hamiltonian paths \eqref{eq:5.4}.
(c) The dissipative Morse meander defined by the label-independent Sturm permutation $\sigma=h_0^{-1}\circ h_1$; see also \eqref{eq:5.5}.
Violating definition \ref{def:5.1}(iii), the original signed regular tetrahedron 3-ball complex is not Sturm.
Therefore the Sturm permutation $\sigma$ necessarily fails to describe the original non-Sturm signed complex (a). 
Instead, $\sigma$ describes a Sturm signed Thom-Smale complex which is not a 3-ball.
}}
\label{fig:5.3}
\end{figure}

Let us bolster those nagging abstract doubts with three specific examples.
Our first example, fig.~\ref{fig:5.2}, recalls the unique Sturm tetrahedron 3-ball with 2+2 faces in the hemispheres $\Sigma_\pm^2 = \Sigma_\pm^2(\mathcal{O})$, alias $\mathbf{W}$ and $\mathbf{E}$; 
see the detailed discussion in \cite{firo3d-3}. 
For such a hemisphere decomposition of the 3-ball tetrahedron, there essentially exists only one signed Thom-Smale complex which complies with all requirements of  definition \ref{def:5.1}; see fig.~\ref{fig:5.2}(a). 
In particular, both, the edge-adjacent location, along the meridian circle, of the poles $\Sigma_\pm^0 = \Sigma_\pm^0(\mathcal{O})$, alias $\mathbf{N}$ and $\mathbf{S}$, and the bipolar orientation are then determined uniquely, up to geometric automorphisms of the tetrahedral complex and trivial equivalences.

In fig.~\ref{fig:5.2}(b) we construct the resulting SZS-pair of Hamiltonian paths $h_0\,,\,h_1$\,.
We follow the practical recipes  of definition \ref{def:5.4}(ii), for the $h_\iota$ predecessors and successors $w_\pm^\iota$ of $\mathcal{O}$, and of definition \ref{def:5.4}(i), for the remaining paths within the respective closed hemispheres.
With the labels $\mathcal{E}=\{1,\ldots, 15\}$ of equilibria in fig.~\ref{fig:5.2}, the resulting paths $h_\iota:\ \{1,\ldots, 15\}\rightarrow \mathcal{E}$ are
	\begin{equation}
	\begin{aligned}
	h_0: \,\, 1\;\; &\text{5 11 6 12 15 14 7 2 10 13 8 3 9 4}\,;\\
	h_1: \,\, 1\;\; &\text{7 2 8 12 6 3 9 11 15 13 10 14 5 4}\,.
	\end{aligned}
	\label{eq:5.2}
	\end{equation}
For the label-independent Sturm permutation $\sigma=h_0^{-1}\circ h_1$ we therefore obtain the dissipative Morse meander of fig.~\ref{fig:5.2}(c), for the 2+2 decomposed Sturm tetrahedron 3-ball:
	\begin{equation}
	\begin{aligned}
	\sigma 
	&=  \lbrace 1,\text{8, 9, 12, 5, 4, 13, 14, 3, 6, 11, 10, 7, 2, 15}\rbrace =\\
	&= \text{(2 8 14) (3 9) (4 12 10 6) (7 13)}\,.
	\end{aligned}
	\label{eq:5.3}
	\end{equation}

Our second example, fig.~\ref{fig:5.3}, starts from a minuscule variation (a) of the same signed tetrahedral 3-ball. 
We only move the South pole $\mathbf{S}$ away from the position 4, which is edge-adjacent to $\mathbf{N}=1$ along the meridian circle. The new, more ``symmetric'' location 3 of $\mathbf{S}$ is not edge-adjacent to $\mathbf{N}$ along the meridian circle.
We keep the 2+2 hemisphere decomposition unchanged, and only adjust the bipolarity of the 1-skeleton accordingly. 
Only the orientation of edge 9, from 3 to 4, has to be reversed to accommodate the misplaced South pole as an orientation sink.
By tetrahedral symmetry, we may keep the orientation of edge 10, from 2 to 4, without loss of generality.
Note however that any orientation of edge 10 now violates the orientation condition (iii) of definition \ref{def:5.1} in the hemisphere $\mathbf{E}=\Sigma_+^2$\,.
All other requirements of definition \ref{def:5.1}, including the overlap condition (iv), remain satisfied.

In fig.~\ref{fig:5.3}(b) we construct the resulting paths $h_0,h_1$ from the practical recipes  of definitions \ref{def:5.3} and \ref{def:5.4}, as before, with the usual labels of equilibria.
This time, we obtain
	\begin{equation}
	\begin{aligned}
	h_0: \,\, 1\;\; &\text{5 14 7 2 10 4 9 11 6 12 15 13 8 3}\,;\\
	h_1: \,\, 1\;\; &\text{7 2 8 12 6 11 15 13 10 14 5 4 9 3}\,.
	\end{aligned}
	\label{eq:5.4}
	\end{equation}
For the Sturm permutation $\sigma=h_0^{-1}\circ h_1$ we therefore obtain the dissipative Morse meander of fig.~\ref{fig:5.2}(c):
	\begin{equation}
	\begin{aligned}
	\sigma 
	&=  \lbrace 1,\text{4, 5, 14, 11, 10, 9, 12, 13, 6, 3, 2, 7, 8, 15}\rbrace =\\
	&= \text{(2 4 14 8 12) (3 5 11) (6 10) (7 9 13)}\,.
	\end{aligned}
	\label{eq:5.5}
	\end{equation}

The Sturm global attractor $\mathcal{A}_f$ which results from that Sturm permutation $\sigma=\sigma_f$\,, however, is not a tetrahedral 3-ball.
In fact, $\mathcal{A}_f$ is not a 3-ball at all.
We prove this indirectly: suppose $\mathcal{A}_f$ is a 3-ball with $\mathcal{O}=15$.
Consider the $h_0$-successor $w_+^0= 13$ of $\mathcal{O}=15$, with Morse index $i(13)=2$; see \eqref{eq:5.4} and fig.~\ref{fig:5.3}(c).
By corollary \ref{cor:4.4} in a 3-ball, the $h_0$-successor 13 of $\mathcal{O}=15$ must coincide with the $h_1$ most distant equilibrium from $\mathcal{O}=15$, in $\mathcal{E}_+^2(\mathcal{O})$. 
Since 2 is even, however, that $h_1$-last equilibrium in $\mathcal{E}_+^2(\mathcal{O})$ is easily identified by its label 14, in fig.~\ref{fig:5.3}(b).
This contradiction shows that $\sigma=\sigma_f$ from \eqref{eq:5.5} is not a tetrahedral 3-ball.
Alternatively to this indirect proof we could also have shown blocking of any heteroclinic orbit from $\mathcal{O}=15$ to the face equilibrium 14, based on zero numbers.

In fact we should have expected such failure: our construction of $h_0, h_1$ in theorem \ref{thm:1.2} is based on a signed cell complex $\mathcal{C}^s$ which is  assumed to be the signed Thom-Smale complex $\mathcal{C}_f^s$ of a Sturm global attractor $\mathcal{A}_f$\,. 
In the example of fig.~\ref{fig:5.3}, we kind of started from a Sturm 3-ball $\mathcal{A}_f$\,, albeit with a misplaced South pole $\mathbf{S}$ and an incorrect signature orientation of the cell $c_9$ in its signed complex.
Cooked up from the recipes of theorem \ref{thm:1.2} and definition \ref{def:5.4} by sheer luck, the Sturm permutation \eqref{eq:5.5} still described a Sturm global attractor $\mathcal{A}_g$, with an associated signed Thom-Smale complex $\mathcal{C}_g^s$\,.
However, $\mathcal{A}_g \neq \mathcal{A}_f$ turns out to be a geometrically different global attractor, and $\mathcal{C}_g^s \neq \mathcal{C}_f^s$\ is not even a 3-ball.

That failure should caution us against another premature temptation. 
It is true that each closed hemisphere $\Sigma_\pm^2(\mathcal{O})$ of any 3-ball Sturm attractor is a planar Sturm attractor, itself.
See definition \ref{def:5.1}(i), (ii).
However it is not true, conversely, that we may glue any pair of planar Sturm disks, with matching poles and boundaries, along the shared meridian boundary to form the 2-sphere boundary $\Sigma^2(\mathcal{O})$ of a Sturm 3-ball attractor.
In fact, the two disks also have to satisfy the mandatory compatibility constraints of definition \ref{def:5.1}(iii), (iv), on edge orientations and overlap, to define a Sturm 3-ball attractor.
In fig.~\ref{fig:5.3}, the two quadrangular closed disks $\mathrm{clos}\, \mathbf{W}$ and $\mathrm{clos}\, \mathbf{E}$, each with a single diagonal, are such a non-matching pair.
Indeed, the Eastern disk $\mathrm{clos}\, \mathbf{E}$ violates definition \ref{def:5.1}(iii).

 \begin{figure}[p!]
\centering \includegraphics[width=1.0\textwidth]{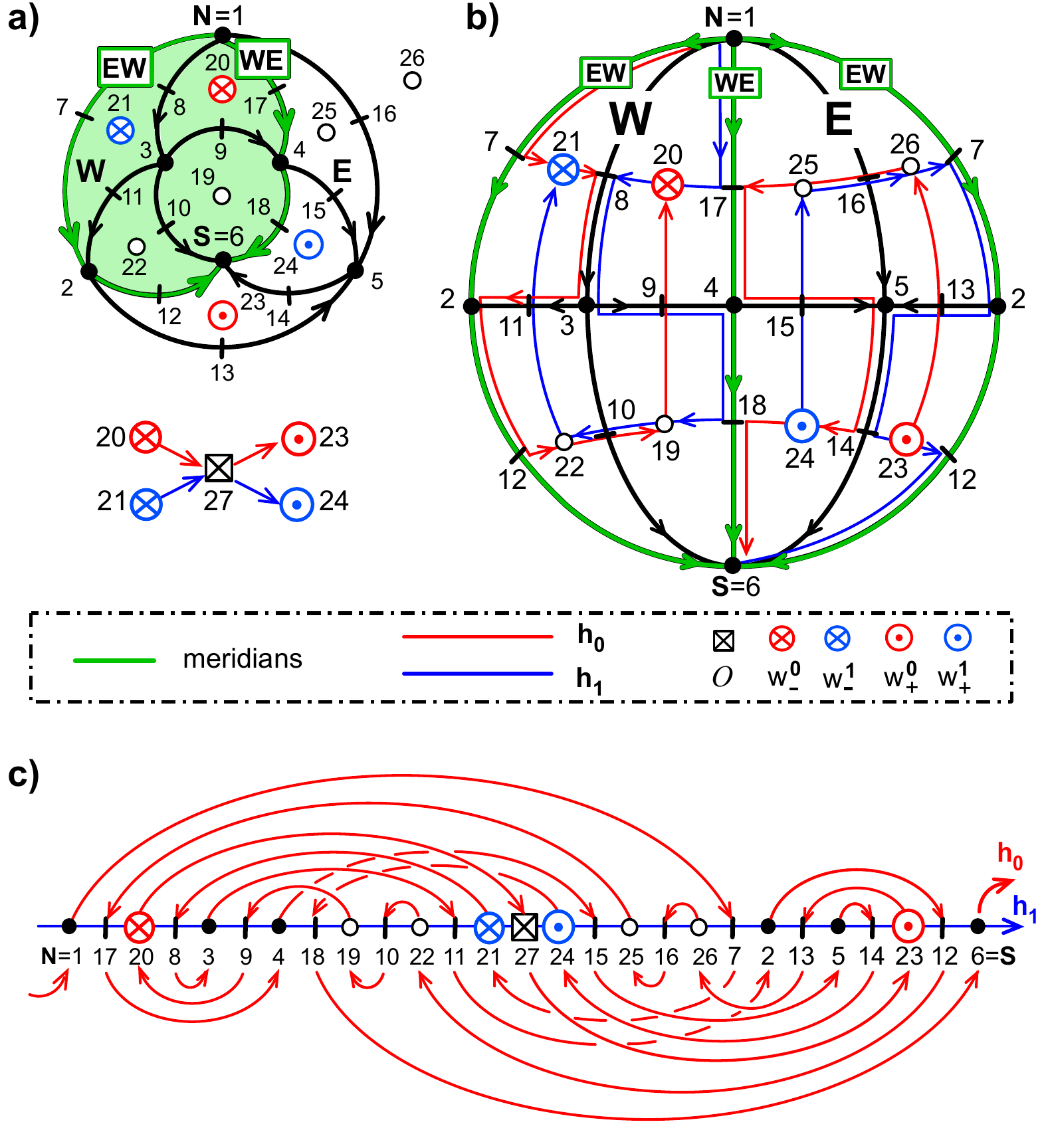}
\caption{\emph{
A signed regular octahedron 3-ball complex with antipodal poles $\mathbf{N}=1$ and $\mathbf{S}=6$. 
By \cite{firo14, firo3d-1, firo3d-3}, there does not exist any Sturm octahedron complex with antipodal poles. 
See figs.~\ref{fig:5.2}, \ref{fig:5.3} for our general setting and notation.
(a) Equilibria $\mathcal{E} =\{1,\ldots, 27\}$, bipolar orientation of the 1-skeleton, and hemisphere decomposition $\mathbf{W}, \mathbf{E}$ into 4+4 faces, one exterior. 
Only the overlap condition (iv) of definition \ref{def:5.1} is violated by the faces of the two pairs $w_-^\iota, w_+^{1-\iota}$\,, respectively.
(b) The SZS-pair $h_0$ (red) and $h_1$ (blue), constructed according to definitions \ref{def:5.3} and \ref{def:5.4}, provides Hamiltonian paths. 
See also \eqref{eq:5.6} for the resulting paths $h_0, h_1$\,.
(c) The involutive permutation $\sigma=h_0^{-1}\circ h_1$ of  \eqref{eq:5.7} is dissipative and Morse, but fails to define a meander. 
There are 16 self-crossings.
Therefore the signed regular octahedron 3-ball complex (a) with antipodal poles fails to define a Sturm signed Thom-Smale complex.
}}
\label{fig:5.4}
\end{figure}

Our third and final example, fig.~\ref{fig:5.4}, applies our path construction to an octahedral 3-ball. 
Fig.~\ref{fig:5.4}(a) prescribes a signed octahedron complex with diagonally opposite poles $\mathbf{N}=1$ and $\mathbf{S}=6$. 
In \cite{firo3d-1, firo3d-3}, however, we have already shown that there does not exist any signed Thom-Smale octahedral complex of Sturm type with diagonally opposite poles. 
See also \cite{firo14} for this surprising phenomenon.
So our construction is asking for trouble, again. 
To be specific we consider a symmetric decomposition into hemispheres $\mathbf{W}, \mathbf{E}$ with 4+4 faces, as indicated in fig.~\ref{fig:5.4}(a). 
The hemisphere splitting avoids direct edges between the meridians $\mathbf{EW}$ and $\mathbf{WE}$, like edge 10 in fig.~\ref{fig:5.3}, which would contradict definition \ref{def:5.1}(iii).
All edge orientations in the bipolar 1-skeleton are then determined uniquely by conditions (i)--(iii) of definition \ref{def:5.1}. 
Only the overlap condition (iv) is violated, this time. 
See fig.~\ref{fig:5.4}(b).

Without difficulties, the practical recipes of definitions \ref{def:5.3} and \ref{def:5.4} provide Hamiltonian paths $h_0,h_1$\,, as before, with the equilibrium labels indicated in fig.~\ref{fig:5.4}(a), (b): 
	\begin{equation}
	\begin{aligned}
	h_0: \text{1 7 21 8 3 11 2 12 22 10 19 9 20 27 23 13 26 16 25 17 4 15 5 14 24 18 6}\,;\\
	h_1: \text{1 17 20 8 3 9 4 18 19 10 22 11 21 27 24 15 25 16 26 7 2 13 5 14 23 12 6}\,.
	\end{aligned}
	\label{eq:5.6}
	\end{equation}
For the permutation $\sigma=h_0^{-1}\circ h_1$ we therefore obtain the involution
	\begin{equation}
	\begin{aligned}
	\sigma 
	&=  \lbrace 1,\text{ 20, 13, 4, 5, 12, 21, 26, 11, 10, 9, 6, 3, 14,}\\
	&\phantom{=\lbrace 1,\ \,}\text{ 25, 22, 19, 18, 17, 2, 7, 16, 23, 24, 15, 8, 27}\rbrace =\\
	&= \text{(2 20) (3 13) (6 12) (7 21) (8 26) (9 11) (15 25) (16 22) (17 19)}\,.
	\end{aligned}
	\label{eq:5.7}
	\end{equation}
This time however, due to the violation of the overlap condition in definition \ref{def:5.1}(iv), the permutation $\sigma$ does not define a meander.
See fig.~\ref{fig:5.4}(c) for the 16 resulting self-crossings generated by the permutation $\sigma$.

In conclusion we see how the recipe of theorem \ref{thm:1.2}, for the construction of the unique Hamiltonian boundary orders $h_0,h_1$ and the unique associated Sturm permutation $\sigma=h_0^{-1}\circ h_1$\,, works well for signed regular complexes $\mathcal{C}^s$ -- provided that these complexes are the signed Thom-Smale complex $\mathcal{C}^s=\mathcal{C}^s_f$ of a Sturm global attractor $\mathcal{A}_f$\,, already. 
In other words, there is a 1-1 correspondence between Sturm permutations and Sturm signed Thom-Smale complexes. 
For non-Sturm signed regular complexes, however, the construction recipe for  $h_0,h_1$ may fail to provide a Sturm permutation $\sigma=h_0^{-1}\circ h_1$\,. 
This was the case for the octahedral example of fig.~\ref{fig:5.4}. 
But even if the construction of a Sturm permutation $\sigma$ succeeds, by our recipe, the result will -- and must -- fail to produce the naively intended Sturm realization of a prescribed non-Sturm signed regular complex $\mathcal{C}^s$.
This was the case for the second tetrahedral example of fig.~\ref{fig:5.3}. 
The goal of a complete geometric description of all Sturm signed Thom-Smale complexes $\mathcal{C}^s_f$\,, as abstract signed regular complexes $\mathcal{C}^s$, therefore requires a precise geometric characterization of the fast unstable manifolds and the Sturmian signatures, on the cell level. 
Only for planar cell complexes and for 3-balls has that elusive goal been reached, so far.


\end{document}